\documentclass{article}%
\usepackage[T1]{fontenc}
\usepackage{tikz}
\usepackage{amsfonts}
\usepackage{amsmath}
\usepackage{amssymb}
\usepackage{graphicx}
\usepackage{color}
\providecommand{\U}[1]{\protect\rule{.1in}{.1in}}
\newenvironment{proof}[1][Proof]{\noindent\textbf{#1.} }{\ \rule{0.5em}{0.5em}}
\newtheorem{theorem}{Theorem}[section]

\numberwithin{equation}{section}

\begin{document}

\title{Second Order Asymptotic Development for the Anisotropic Cahn-Hilliard Functional}
\author{Gianni Dal Maso\\SISSA, \\Via Bonomea 265, 34136 Trieste, Italy
\and Irene Fonseca\\Department of Mathematical Sciences,\\Carnegie Mellon University,\\Pittsburgh PA 15213-3890, USA
\and Giovanni Leoni\\Department of Mathematical Sciences, \\Carnegie Mellon University, \\Pittsburgh PA 15213-3890, USA}
\maketitle

\begin{abstract}\noindent The asymptotic behavior of an anisotropic Cahn-Hilliard functional with prescribed mass and Dirichlet boundary 
condition is studied when the parameter $\varepsilon$ 
 that determines the width of the transition layers tends to zero. 
The  double-well potential is assumed to be even and equal to $|s-1|^\beta$ near $s=1$, with  $1<\beta<2$.  
The first order term  in the asymptotic development by
$\Gamma$-convergence is well-known, and is related to a suitable anisotropic perimeter of the interface. Here it is shown that, 
under these assumptions, the second order term is zero, which gives an estimate on the rate of convergence of the minimum values.
\end{abstract}





\section{Introduction}

In this paper we study the second order term in the asymptotic development by
$\Gamma$-convergence for the anisotropic Cahn-Hilliard functional (see, e.g., \cite{MM}, 
\cite{Gurtin85}, \cite{Mod87}, \cite{sternberg88}, \cite{fonseca-tartar}, \cite{bouchitte},  \cite{barroso-fonseca})%
\begin{equation}
\mathcal{W}_{\varepsilon}\left(  u\right)  :=\int_{\Omega}\left(  W\left(
u\left(  x\right)  \right)  +\varepsilon^{2}\Phi^{2}\left(  \nabla u\left(
x\right)  \right)  \right)  \,dx\,, \label{cahn-hilliard functional}%
\end{equation}
where $\Omega$ is a bounded open set in $\mathbb{R}^{n}$, $n\ge 2$, with Lipschitz
boundary. Here  $W:\mathbb{R}\rightarrow\left[  0,+\infty\right)  $ is an even
function of class $C^{1}$ such that $W\left(  s\right)  =0$ if and only if
$s=\pm1$, with $W(s)=|s-1|^\beta$ near $s=1$ for some $1<\beta<2$, and $\Phi:\mathbb{R}^{n}\rightarrow\left[  0,+\infty\right)$ is convex, even,  and
positively homogeneous of degree one.

We impose a mass constraint and a boundary condition: 
\begin{equation}
u\in H^{1}\left(  \Omega\right)  \,,\quad\int_{\Omega}u\left(  x\right)
\,dx=m\,,\quad\text{and\quad}u=1\text{ on }\partial\Omega\,,
\label{constraint}%
\end{equation}
where $m$ is a prescribed constant satisfying the inequalities
\begin{equation}
-\left\vert \Omega\right\vert <m<\left\vert \Omega\right\vert \,.
\label{inequalities m}%
\end{equation}

Given  a sequence of functionals $F_{\varepsilon}:X\rightarrow\left(
-\infty,\infty\right]  $  defined on a metric space $X$, we write
$F_\varepsilon \overset{\Gamma}{\rightarrow} F^{(0)}$ if $\{F_{\varepsilon} \}$ $\Gamma$-converges to $F^{(0)}$, as $\varepsilon\to 0+$,  with respect to the metric topology of $X$. We recall the notion of \emph{asymptotic development by $\Gamma$-convergence of
order $k$}: 
\[
F_{\varepsilon}\overset{\Gamma}{=}F^{\left(  0\right)  }+\varepsilon
F^{\left(  1\right)  }+\cdots+\varepsilon^{k}F^{\left(  k\right)  }+o\left(
\varepsilon^{k}\right)
\]
if $F_{\varepsilon}\overset{\Gamma}{\rightarrow} F^{(0)}$ and 
\begin{equation}
F_{\varepsilon}^{\left(  i\right)  }:=\dfrac{F_{\varepsilon}^{\left(
i-1\right)  }-\inf\nolimits_{X}F^{\left(  i-1\right)  }}{\varepsilon
}
\overset{\Gamma}{\rightarrow}F^{\left(  i\right)  } \label{order i}%
\end{equation}
for $i=1,\ldots,k$, where $F_{\varepsilon}^{\left(  0\right)  }%
:=F_{\varepsilon}$ (see \cite{anzellotti-baldo93},
\cite{anzellotti-baldo-orlandi96}, \cite[Section 1.10]{Braid2002}).

For the sequence of functionals (\ref{cahn-hilliard functional}) we take $X:=
L^{1}\left(  \Omega\right)  $ and we set $\mathcal{W}_{\varepsilon}\left(
u\right)  :=+\infty$ if 
 (\ref{constraint}) is not satisfied. The
zero order term is
\begin{equation}
\mathcal{W}^{(0)}\left(  u\right)  :=\int_{\Omega}W\left(  u\left(
x\right)  \right)  \,dx \nonumber
\end{equation}
if the mass constraint in \eqref{constraint} is satisfied and $\mathcal{W}^{(0)}\left(  u\right)  :=+\infty$ otherwise. The $\Gamma$-liminf inequality is a consequence of Fatou's Lemma. The
$\Gamma$-limsup inequality is straightforward. 

Note that $\inf\nolimits_{X}%
\mathcal{W}^{(0)}=0$ and the minimizers are given by all functions of the form
$u_{E}:=1-2\chi_{E}$, where $E$ is an arbitrary measurable subset of $\Omega$
satisfying the volume constraint
\begin{equation}
\left\vert E\right\vert =\frac{\left\vert \Omega\right\vert -m}{2}=:V_{m} \,,
\label{volume constraint E}%
\end{equation}
which is equivalent to the mass constraint in (\ref{constraint}) for $u_{E}$.  Here, and in what follows, $\chi_{E}$
is the characteristic function of $E$ defined by $\chi_{E}:=1$ on $E$ and
$\chi_{E}:=0$ on $\Omega\setminus E$

To study the first order term for (\ref{cahn-hilliard functional}), we
introduce the rescaled functionals defined by
\begin{equation}
\mathcal{F}_{\varepsilon}\left(  u\right)  :=\int_{\Omega}\left(  \frac
{1}{\varepsilon}W\left(  u\left(  x\right)  \right)  +\varepsilon\Phi
^{2}\left(  \nabla u\left(  x\right)  \right)  \right)  \,dx
\nonumber
\end{equation}
if  (\ref{constraint}) is satisfied. We
extend $\mathcal{F}_{\varepsilon}$ to $L^{1}\left(  \Omega\right)  $ by
setting $\mathcal{F}_{\varepsilon}\left(  u\right)  :=+\infty$ if  (\ref{constraint})
is not satisfied.

By adapting well-known arguments developed in \cite{barroso-fonseca},
\cite{bouchitte}, \cite{MM}, \cite{Mod87}, \cite{sternberg88}, it can be shown
(see Theorem \ref{theorem gamma convergence} below) that the first order term
$\mathcal{W}^{(1)}$ for (\ref{cahn-hilliard functional}), which by
(\ref{order i}) coincides with the $\Gamma$-limit of $\left\{  \mathcal{F}%
_{\varepsilon}\right\}  $, is given by%
\begin{equation}
\mathcal{F}_{0}\left(  u\right)  :=c_{W}\operatorname{P}_{\Phi}\left(
E\right)  \label{functional F 0}%
\end{equation}
if
\begin{equation}
u=u_{E}:=1-2\chi_{E}\,,\quad E\subset\Omega\,,\quad\operatorname{P}\left(  E\right)
<+\infty\,,\,\text{ and}\quad \left\vert E\right\vert =V_m\,\,,
\label{constraint inequality 2*}%
\end{equation}
while $\mathcal{F}_{0}\left(  u\right)
:=+\infty$ if \eqref{constraint inequality 2*} is not satisfied.  Here 
\begin{equation}
c_{W}:=2\int_{-1}^{1}\sqrt{W\left(  s\right)  }\,ds \label{cW}%
\end{equation}
and $\operatorname{P}_{\Phi}$ is the $\Phi$-perimeter, defined for every
$E\subset\mathbb{R}^{n}$ with finite perimeter by%
\begin{equation}
\operatorname{P}_{\Phi}\left(  E\right)  :=\int_{\partial^{\ast}E}\Phi\left(
\nu_{E}(x)\right)  \,d\mathcal{H}^{n-1}(x)\,, \label{perimeter gamma}%
\end{equation}
where $\partial^{\ast}E$ is the reduced boundary of $E$, $\nu_{E}$ is the
measure theoretic outer unit normal of $E$, and $\mathcal{H}^{n-1}$ is the
$(n-1)$-dimensional Hausdorff measure. Observe that in contrast with the results in the literature 
just quoted, due to the boundary condition in \eqref{constraint} in \eqref{functional F 0} we obtain  the full $\Phi$-perimeter of $E$ as opposed to the relative $\Phi$-perimeter of $E$ in $\Omega$.

The main goal of this paper is to study the second order term $\mathcal{W}%
^{(2)}$ for (\ref{cahn-hilliard functional}). Under some additional assumptions on $\Omega$ and $W$ (see \eqref{W prime},  \eqref{lim W prime}, \eqref{exists},  and  \eqref{enlarged ball} in Section \ref{s:polars}), we prove that
$\mathcal{W}^{(2)}\left(  u \right)  =0$ if $u$ is a minimizer of
$\mathcal{F}_{0}$ and $\mathcal{W}^{(2)}\left(  u \right)  =+\infty$
otherwise. The second assertion is trivial. By (\ref{order i}) the first
assertion amounts to proving the following properties:

\begin{itemize}
\item[(a)] ($\Gamma$-liminf inequality) for every sequence $\left\{
u_{\varepsilon}\right\}  \subset H^{1}\left(  \Omega\right)  $ satisfying
(\ref{constraint}) and converging strongly in
$L^{1}(\Omega)$ to a minimizer $u_{0}$ of $\mathcal{F}_{0}$, we have
\begin{equation}
\liminf_{\varepsilon\rightarrow0+}\frac{\mathcal{F}_{\varepsilon}\left(
u_{\varepsilon}\right)  -\mathcal{F}_{0}\left(  u_{0}\right)  }{\varepsilon
}\geq0\,; 
\label{second gamma liminf}%
\end{equation}

\item[(b)] ($\Gamma$-limsup inequality) for every minimizer $u_{0}$ of
$\mathcal{F}_{0}$ there exists a sequence $\left\{  u_{\varepsilon}\right\}
\subset H^{1}\left(  \Omega\right)  $ converging strongly to $u_{0}$ in $L^{1}(\Omega
)$, satisfying  (\ref{constraint}) and 
such that
\begin{equation}
\limsup_{\varepsilon\rightarrow0+}\frac{\mathcal{F}_{\varepsilon}\left(
u_{\varepsilon}\right)  -\mathcal{F}_{0}\left(  u_{0}\right)  }{\varepsilon
}\leq0\,. \label{second gamma limsup}%
\end{equation}

\end{itemize}

By standard properties of $\Gamma$-convergence the inequalities (a) and (b) imply that 
$$\min \mathcal{F}_{\varepsilon}=\min \mathcal{F}_{0}+o(\varepsilon)=c_W\operatorname{P}_{\Phi}(E_0)+o(\varepsilon)\,,$$
where $E_0$ is a minimizer of $\operatorname{P}_{\Phi}$ under the constraint \eqref{constraint inequality 2*}, 
which gives
$$\min \mathcal{W}_{\varepsilon}=\varepsilon c_W\operatorname{P}_{\Phi}(E_0)+o(\varepsilon^2)\,.$$

A similar problem was studied in \cite{anzellotti-baldo-orlandi96} for the single-well potential $W(s)=s^2$ without  imposing the mass constraint  and assuming a strictly positive boundary condition $g$. This forces a transition near $\partial\Omega$ and leads to a second order  term $\mathcal{W}%
^{(2)}$ in the asymptotic expansion of the form $\frac{1}{2}\int_{\partial \Omega}g^2K\,d\mathcal{H}^{n-1}$, where $K$ is the mean curvature of $\partial\Omega$.

We conclude by discussing  our hypotheses. The assumption that  $W$ is even  
is used in a crucial way to cancel many terms in the estimates due to symmetry arguments. 
The hypothesis that  $W(s)=|s-1|^\beta$ near $s=1$ for some $1<\beta<2$ is also important. 
Indeed, in the case $\beta=2$ and without assuming the boundary condition in  \eqref{constraint}, it can be shown that  the second order term $\mathcal{W}%
^{(2)}$ in the asymptotic expansion may be different from zero (see \cite{leoni-murray}). 

Finally, we observe that the case $n=1$ is completely different, since the minimizers of $\operatorname{P}_{\Phi}$ under the constraint \eqref{constraint inequality 2*} are intervals and so the geometry plays no role. However, different nontrivial issues have been addressed (see, e.g., in  \cite{bellettini} and \cite{carr-gurtin-slemrod1984}, and also \cite{anzellotti-baldo93}).

\section{Preliminaries}

\label{s:polars}

Let $W:\mathbb{R}\rightarrow\mathbb{R}$ be a double well potential of class
$C^{1}$ such that $W\geq0$ and $W\left(  s\right)  =0$ if and only if $s=\pm
1$. Assume, in addition, that%
\begin{align}
&W\left(  s\right)  =W\left(  -s\right)  \,, \label{W even}\\
&W^\prime(s)>0 \quad \text{for }s>1\,, \label{W prime}\\
&\liminf_{s\to+\infty}W'(s)>0\,,\label{lim W prime}
\end{align}
and that there exist two constants
$0<a<1$ and $1<\beta<2$ such that%
\begin{equation}
W\left(  s\right)  = \left|  s-1\right|  ^{\beta}\quad\text{for }1-a\leq
s\leq1+a\,. \label{W beta}%
\end{equation}

Let $z$  be the unique global solution with values in $[-1,1]$ of the Cauchy problem%
\begin{equation}
z^{\prime}(t)=\sqrt{W\left(  z(t)\right)  }\,,\quad z\left(  0\right)  =0\,.
\label{Cauchy problem}%
\end{equation}
A rescaled version of this function will play an important role in the construction of the
recovery sequence (see \eqref{hat u epsilon}) for the $\Gamma$-limsup inequality \eqref{second gamma limsup} .  For this reason, the function $z$ will be called the ``optimal profile'' of the phase transition.

For $-1<z\left(  t\right)  <1$ we obtain, by integration,
\[
t=\int_{0}^{z\left(  t\right)  }\!\!\!\! \frac{ds}{\sqrt{W\left(  s\right)  }%
}\,.
\]
It follows that $z$ is odd and $z(t)=1$ for all $t\ge \tau_W$, where
\begin{equation}
\tau_{W}:=\int_{0}^{1}\! \frac{ds}{\sqrt{W\left(  s\right)  }}\,,
\label{tau W}%
\end{equation}
which is finite thanks to (\ref{W beta}) since $\beta<2$. Moreover $-1<z(t)<1$ for $-\tau
_{W}<t<\tau_{W}$. 

Define
\begin{equation}
c_{W}:=\int_{-\tau_{W}}^{\tau_{W}}\left(  W\left(  z\left(  t\right)  \right)
+\left\vert z^{\prime}\left(  t\right)  \right\vert ^{2}\right)  \,dt\,.
\label{constant cW}%
\end{equation}
Note that by (\ref{Cauchy problem}) we have%
\begin{equation}
\int_{-\tau_{W}}^{\tau_{W}}W\left(  z\left(  t\right)  \right)  \,dt=\int_{-\tau_{W}}^{\tau_{W}}\left\vert z^{\prime}\left(  t\right)  \right\vert ^{2}\,dt=\frac{c_{W}}%
{2}\,, 
\label{constant cW 2}%
\end{equation}
therefore (\ref{cW}) holds.
It can be shown that for every $b\ge \tau_W$ the function $z$ is the unique solution of the minimum problem
\begin{equation}\label{minimum cW}
\min_{w\in H^{1}\left(  -b,b\right)  \,, \,w\left(  0\right)  =0}\int%
_{-b}^{b}\left(  W\left(  w\left(  s\right)  \right)  +\left\vert w^{\prime
}\left(  s\right)  \right\vert ^{2}\right)  \,ds=c_{W}\,.
 \end{equation}

Let $\Phi:\mathbb{R}^{n}\rightarrow\mathbb{R}$, $n\ge 2$, be convex, even, and positively
homogeneous of degree one, such that
\begin{equation}
c_{\Phi} \left\vert \xi\right\vert \leq\Phi\left(  \xi\right)  \leq C_{\Phi
}\left\vert \xi\right\vert \label{growth Gamma}%
\end{equation}
for all $\xi\in\mathbb{R}^{n}$ and for some $C_{\Phi}\geq c_{\Phi} >0$, where
$\left\vert \cdot\right\vert $ is the Euclidean norm in $\mathbb{R}^{n}$.

The polar function $\Phi^{\circ}$ is defined by
\[
\Phi^{\circ}\left(  \eta\right)  :=\sup_{\xi\neq0} \frac{\eta\cdot\xi}%
{\Phi\left(  \xi\right)  }
\]
for every $\eta\in\mathbb{R}^{n}$. It turns out that $\Phi^{\circ}$ is convex, even, positively homogeneous of degree one on $\mathbb{R}^{n}$ (see \cite{R}), and 
\begin{equation}\label{polar diff}
\Phi(\nabla\Phi^{\circ}\left(\eta\right))=1\quad\text{ for a.e. }\eta\in\mathbb{R}^n\,.
\end{equation}

Moreover, it
satisfies the inequalities
\[
\frac{1}{C_{\Phi}} \left\vert \eta\right\vert \leq\Phi^{\circ}\left(
\eta\right)  \leq\frac{1}{c_{\Phi}}\left\vert \eta\right\vert
\]
for every $\eta\in\mathbb{R}^{n}$.

The ball with respect to the norm $\Phi^{\circ}$ centered at $x_0\in\mathbb{R}^{n}$ and with radius $\rho>0$ is denoted by
\begin{equation}
B^{\Phi^{\circ}}_{\rho}\!\!\left(  x_{0} \right)  :=\left\{  x\in\mathbb{R}%
^{n}:\,\Phi^{\circ}\!\left(  x-x_{0}\right)  < \rho\right\}  \,,
\label{Wulff set}%
\end{equation}
Observe that 
\begin{align}
\left\vert B^{\Phi^{\circ}}_{\rho}\!\!\left(  x_{0} \right)  \right\vert  &
=\kappa_{\Phi}\rho^{n}\,,
\nonumber
\\
\operatorname{P}_{\Phi}\!\left(  B^{\Phi^{\circ}}_{\rho}\!\!\left(  x_{0}
\right)  \right)   &  =n\kappa_{\Phi}\rho^{n-1}\,, \label{perimeter rho}%
\end{align}
where  $\operatorname{P}_{\Phi}$ is the $\Phi$-perimeter introduced in (\ref{perimeter gamma}), and
\begin{equation}
\kappa_{\Phi}:=\left\vert B^{\Phi^{\circ}}_{1}\!\!\left(  0 \right)
\right\vert \,. \label{kappa Phi}%
\end{equation}
 It is easy to check (see, e.g., \cite{AFTL,Fon}) that for every measurable function $w:\left[  0, R\right]
\rightarrow[0,+\infty]$ we have
\begin{equation}
\int_{B^{\Phi^{\circ}}_{R}\!\left(  x_{0} \right)  } w\left(  \Phi^{\circ
}\left(  x-x_{0}\right)  \right)  \,dx =n\kappa_{\Phi}\int_{0}^{R} w\left(
\rho\right)  \,\rho^{n-1} d\rho\,, \label{integration polar}%
\end{equation}
Moreover, if $\pm w$ is nondecreasing and the composite function $v(x):= w(
\Phi^{\circ}(x-x_0))$ belongs to $H^{1}\left(  B^{\Phi^{\circ}}_{R}\!\!\left(  x_{0}
\right)  \right)  $, then (see \eqref{polar diff})
\begin{equation}
\Phi\left(  \nabla v \left(  x\right)  \right)  =\pm w^{\prime}\left(
\Phi^{\circ} \left(  x-x_{0}\right)  \right)  \quad\text{for a.e.}\ x\in
B^{\Phi^{\circ}}_{R}\!\!\left(  x_{0} \right)  \,. \label{derivative polar}%
\end{equation}

The geometry of the minimizers of $\operatorname{P}_{\Phi}$ in $\mathbb{R}^N$ with prescribed volume $V>0$ is well-known. Indeed, it was established in \cite{Fon} and \cite{FonMu} (see also \cite{Tay1},
\cite{Tay2}, \cite{Tay3}) that the minimum of the problem
\begin{equation}
\min\left\{  \operatorname{P}_{\Phi}\left(  E\right)  :\,E\subset\mathbb{R}^n\text{ with finite
perimeter, }\left\vert E\right\vert =V\right\}  
\nonumber
\end{equation}
is attained by  all balls \eqref{Wulff set} centered at an arbitrary point $x_{0}\in\mathbb{R}^n$ and with radius $\rho>0$ chosen so that $\left\vert
B^{\Phi^{\circ}}_{\rho}\!\!\left(  x_{0} \right)  \right\vert =V$.  These balls are called  
 \emph{Wulff sets} for the $\Phi$-perimeter after the pioneering work of Wulff 
\cite{Wulff}.

Let $\Omega$ be a bounded open set of $\mathbb{R}^{n}$ with Lipschitz
boundary, let $m\in\mathbb{R}$ be a constant satisfying  \eqref{inequalities m}, and let $V_m$ be defined by \eqref{volume constraint E}.
 
We are interested in the minimum problem
\begin{equation}
\min\left\{  \operatorname{P}_{\Phi}\left(  E\right)  :\,E\text{ set of finite
perimeter, }E\subset\Omega, \left\vert E\right\vert =V_{m}\right\}\,.
\label{min perimeter Omega}%
\end{equation}
Let $r>0$ be such that $\left\vert B^{\Phi^{\circ}%
}_{r}\!\!\left(  0 \right)  \right\vert =V_{m}$, that is, \begin{equation}
r:=\left(  \frac{\left\vert \Omega\right\vert -m}{2\kappa_{\Phi}} \right)
^{\!\!1/n}\,, \label{r}%
\end{equation}
which gives
\begin{equation}
m=\left\vert \Omega\right\vert - 2\kappa_{\Phi} r^{n}=\left\vert \Omega\right\vert -2\left\vert B^{\Phi^{\circ}}_{r}\!\!\left( 0 \right)\right\vert \,.
\label{constraint reached}%
\end{equation}
A minimizer of \eqref{min perimeter Omega} is attained at any Wulff set $B^{\Phi^{\circ}}_{r}\!\!\left(  x_{0} \right)  $ contained in $\Omega$, provided there is at least one. For this reason we assume that 
 there exists $y_{0}\in\Omega$ such that%
\begin{equation}
B^{\Phi^{\circ}}_{r}\!\!\left(  y_{0} \right)  \subset\Omega\,.\label{exists}
\end{equation}
Our results are strongly hinged on this assumption. 

We observe that there exists  a Wulff set contained in $\Omega$ provided that $m$ is close 
to $|\Omega|$, which corresponds to $V_m$ and $r$ sufficiently small. 

For technical reasons, related to the proof of the $\Gamma$-limsup inequality, we further assume 
 that, whenever $B^{\Phi^{\circ}%
}_{r}\!\!\left(  x \right)  \subset\Omega$ for  $x\in\Omega$, then there exist $y\in\Omega$
and $\delta>0$ with
\begin{equation}
B^{\Phi^{\circ}}_{r}\!\!\left(  x \right)  \subset B^{\Phi^{\circ}%
}_{r+\delta}\!\left(  y \right)  \subset\Omega\,. \label{enlarged ball}%
\end{equation}

Given $\varepsilon>0$, we define
\begin{equation}
\mathcal{E}_{\varepsilon}\left(  u\right)  :=\int_{\Omega}\left(  \frac
{1}{\varepsilon}W\left(  u\left(  x\right)  \right)  +\varepsilon\Phi
^{2}\left(  \nabla u\left(  x\right)  \right)  \right)  \,dx
\label{functional E epsilon omega}%
\end{equation}
for $u\in H^{1}\left(  \Omega\right)$. 
Some arguments in what follows will require a localization of this energy, i.e., for every bounded open set
$A$ of $\mathbb{R}^{n}$ with Lipschitz boundary and for every $\varepsilon>0$, 
we define%
\begin{equation}\label{E local}
 \mathcal{E}_{\varepsilon}\left(  u,A\right)  :=\int_{A}\left(  \frac
{1}{\varepsilon}W\left(  u\left(  x\right)  \right)  +\varepsilon\Phi
^{2}\left(  \nabla u\left(  x\right)  \right)  \right)  \,dx
\end{equation}
if $u\in H^{1}\left(  A\right)  $ and $\mathcal{E}_{\varepsilon}\left(  u,A\right)  :=+\infty$ if $u\in L^{1}\left(  A\right)  \setminus H^{1}\left(
A\right)  $. Note that $\mathcal{E}_{\varepsilon}\left(  \cdot,\Omega\right)=
\mathcal{E}_{\varepsilon}(  \cdot)$ as defined in \eqref{functional E epsilon omega}.

Consider the constrained functionals defined on $L^1(\Omega)$ by
\begin{equation}
\mathcal{F}_{\varepsilon}\left(  u\right)  :=\begin{cases}
\mathcal{E}_{\varepsilon}\left(  u\right)& \text{if }u \text{ satisfies \eqref{constraint},} \\
+\infty& \text{otherwise}
\end{cases}
\nonumber
\end{equation}
and 
\begin{equation}
\mathcal{F}_{0}\left(  u\right)  :=\begin{cases}
c_{W}\operatorname{P}_{\Phi}\left(
E\right)& \text{if }u \text{ satisfies \eqref{constraint inequality 2*},} \\
+\infty& \text{otherwise,}
\end{cases}
\label{functional F zero}%
\end{equation}
where $c_{W}$ is defined in (\ref{constant cW}). 
 It is
important to observe that in (\ref{functional F zero}) the $\Phi$-perimeter
$\operatorname{P}_{\Phi}\left(  E\right)  $ is defined by integrating over all
the reduced boundary $\partial^{\ast}E$ of $E$ and not only on $\Omega
\cap\partial^{\ast}E$, i.e.,  we consider the $\Phi$-perimeter in
$\mathbb{R}^{n}$ and not the relative $\Phi$-perimeter in $\Omega$.

A minimizer of $\mathcal{F}_{0}$ is a function of the form $u_{0}%
=1-2\chi_{E_{0}}$, where $E_{0}\subset\Omega$ is a minimizer of \eqref{min perimeter Omega}. 
Hence,  $E_{0}$ has the form
$B^{\Phi^{\circ}}_{r}\!\!\left(  x_{0} \right)  $, with $B^{\Phi^{\circ}}%
_{r}\!\!\left(  x_{0} \right)  \subset\Omega$ and $r$ defined by (\ref{r}).
Then (\ref{perimeter rho}) gives
\begin{equation}
\operatorname{P}_{\Phi}\left(  E_{0}\right)  = \operatorname{P}_{\Phi}\left(
B^{\Phi^{\circ}}_{r}\!\!\left(  x_{0} \right)  \right)  =n\kappa_{\Phi}
r^{n-1}\,. \label{minimal perimeter value}%
\end{equation}

We now state the main result of this section.

\begin{theorem}
\label{theorem gamma convergence}The family $\left\{  \mathcal{F}%
_{\varepsilon}\right\}  $ $\Gamma$-converges in $L^{1}\left(  \Omega\right)  $
to $\mathcal{F}_{0}$.
\end{theorem}

If the boundary condition $u=1$ on $\partial\Omega$ is omitted and
$\operatorname{P}_{\Phi}\left(  E\right)  $ is replaced by by
$\operatorname{P}_{\Phi}\left(  E, \Omega\right)  :=\int_{\Omega\cap
\partial^{\ast}\!E}\Phi\left(  \nu_{E}\right)  \,d\mathcal{H}^{n-1}$, this
result has been established in \cite{MM}, \cite{Mod87}, \cite{sternberg88} for
the isotropic scalar-valued case, in \cite{fonseca-tartar} for the isotropic
vector-valued case, in \cite{bouchitte}, \cite{owen-sternberg} for the
anistropic, scalar-valued case, and in \cite{barroso-fonseca} for the
anisotropic, vector-valued case (see also \cite{Braid}). In the proof below we show
how to take into account the boundary condition.

\bigskip

\begin{proof}
[Proof of Theorem \ref{theorem gamma convergence}] 
Similarly to \eqref{functional E epsilon omega} and \eqref{E local}, we localize \eqref{functional F zero} as
\[
\mathcal{E}_{0}\left(  u,A\right)  :=c_{W}\operatorname{P}_{\Phi}\left(
E,A\right)
\]
if
\begin{equation}
u=1-2\chi_{E}\,,\text{ with }E\subset A\text{ with finite perimeter in }A\,, 
\label{u in A}%
\end{equation}
where%
\[
\operatorname{P}_{\Phi}\left(  E,A\right)  :=\int_{A\cap\partial^{\ast}E}%
\Phi\left(  \nu_{E}(x)\right)  \,d\mathcal{H}^{n-1}(x)%
\]
is the the relative $\Phi$-perimeter of $E$ in $A$. We extend 
$\mathcal{E}_{0}\left(  \cdot,A\right)$ to $L^{1}\left(  A\right)  $ 
by setting $\mathcal{E}_{0}\left(  u,A\right)  :=+\infty$ if (\ref{u in A}) is not satisfied. By \cite[Theorem
3.5(i)]{bouchitte} the family $\left\{  \mathcal{E}_{\varepsilon}\left(  \cdot,A\right)\right\}
$ $\Gamma$-converges in $L^{1}\left(  A\right)  $ to $\mathcal{E}_{0}\left(  \cdot,A\right)$.

Fix a sequence $\varepsilon_{k}\rightarrow0^{+}$ and define%
\[
\mathcal{F}_{0}^{\prime}:=\Gamma\text{-}\liminf_{k\rightarrow\infty
}\mathcal{F}_{\varepsilon_{k}}\quad\text{and}\quad\mathcal{F}_{0}%
^{\prime\prime}:=\Gamma\text{-}\limsup_{k\rightarrow\infty}\mathcal{F}%
_{\varepsilon_{k}}\,.
\]
We prove that $\mathcal{F}_{0}\leq\mathcal{F}_{0}^{\prime}$. Let $u\in
L^{1}\left(  \Omega\right)  $ be such that $\mathcal{F}_{0}^{\prime}\left(
u\right)  <+\infty$. Then there exists a sequence $\left\{  u_{k}\right\}  $
converging to $u$ strongly in $L^{1}\left(  \Omega\right)  $ and such that%
\[
  \liminf_{k\rightarrow\infty
}\mathcal{F}_{\varepsilon_{k}}\left(  u_{k}\right)=\mathcal{F}_{0}^{\prime}\left(  u\right)<+\infty  \,.
\]
Passing to a subsequence, not relabeled, we may assume that the liminf is a
limit and that $\mathcal{F}_{\varepsilon_{k}}\left(  u_{k}\right)  <+\infty$
for every $k$. By (\ref{constraint}) we have that $u_{k}\in H^{1}\left(
\Omega\right)  $, $\int_{\Omega}u_{k}\left(  x\right)  \,dx=m$, and $u_{k}=1$
on $\partial\Omega$. Fix a bounded open set $A$ of $\mathbb{R}^{n}$ with
Lipschitz boundary such that $\overline{\Omega}\subset A$, and extend $u_{k}$
and $u$ to $A$ by setting $u_{k}=u=1$ on $A\setminus\Omega$. Then $u_{k}\in
H^{1}\left(  A\right)  $ and $\left\{  u_{k}\right\}  $ converges to $u$ strongly in
$L^{1}\left(  A\right)  $. Hence, by \cite[Theorem 3.5(i)]{bouchitte},
\begin{equation}
\mathcal{E}_{0}\left(  u,A\right)  \leq\lim_{k\rightarrow\infty}%
\mathcal{E}_{\varepsilon_k}\left(  u_k,A\right) <+\infty\,.
\label{bou 1}%
\end{equation}
Therefore, there exists a set $E\subset A$ with finite perimeter such that
$u=1-2\chi_{E}$ in $A$. Since $u=1$ in $A\setminus\Omega$, it follows that
$E\subset\Omega$ up to a set of measure zero. Hence, $\mathcal{E}_{0}\left(  u,A\right)%
  =\operatorname{P}_{\Phi}\left(  E,A\right)
=\operatorname{P}_{\Phi}\left(  E\right)  =\mathcal{F}_{0}\left(  u\right)  $.
On the other hand, $\mathcal{E}_{\varepsilon_{k}}\left(  u_{k},A\right)
=\mathcal{F}_{\varepsilon_{k}}\left(  u_{k}\right)  $, since $W\left(
u_{k}\right)  =W\left(  1\right)  =0$ and $\Phi^{2}\left(  \nabla
u_{k}\right)  =\Phi^{2}\left(  0\right)  =0$ in $A\setminus\Omega$. Together
with (\ref{bou 1}), this shows that%
\[
\mathcal{F}_{0}\left(  u\right)  \leq\liminf_{k\rightarrow\infty}%
\mathcal{F}_{\varepsilon_{k}}\left(  u_{k}\right)  =\mathcal{F}_{0}^{\prime
}\left(  u\right)  \,.
\]
This concludes the proof of the inequality $\mathcal{F}_{0}\leq\mathcal{F}%
_{0}^{\prime}$.

We now prove that $\mathcal{F}_{0}^{\prime\prime}\leq\mathcal{F}_{0}$. Let
$u\in L^{1}\left(  \Omega\right)  $ be such that $\mathcal{F}_{0}%
\left(  u\right)  <+\infty$. By (\ref{constraint inequality 2*}%
) there exists a set $E\subset\Omega$ with finite perimeter such that
$u=1-2\chi_{E}$. Since $\Omega$ has a Lipschitz boundary, there exists a
sequence of sets $\left\{  E_{j}\right\}  $ of finite perimeter such that
$\chi_{E_{j}}\rightarrow\chi_{E}$ in $L^{1}\left(  \Omega\right)  $,
$\operatorname{P}_{\Phi}\left(  E_{j}\right)  \rightarrow\operatorname{P}%
_{\Phi}\left(  E\right)  $, $E_{j}\subset\subset\Omega$, and $\left\vert
E_{j}\right\vert =\left\vert E\right\vert $ for every $j$. One way to
construct $\left\{  E_{j}\right\}  $ is to consider a sequence $\left\{
r_{j}\right\}  $ of retractions $r_{j}:\mathbb{R}^{n}\rightarrow\mathbb{R}%
^{n}$ of class $C^{1}$ such that $\operatorname*{supp}\left(  r_{j}%
-\operatorname*{id}\right)  \subset\subset\mathbb{R}^{n}$, $r_{j}%
-\operatorname*{id}\rightarrow0$ in $C_{c}^{1}\left(  \mathbb{R}%
^{n};\mathbb{R}^{n}\right)  $, and $r_{j}\left(  \overline{\Omega}\right)
\subset\Omega$, where $\operatorname*{id}$ is the identity map (for the
existence of these retractions see, e.g., \cite[Proposition 1.2]%
{dalmaso-musina1989}). For $j$ large enough $r_{j}$ is invertible and
$r_{j}^{-1}-\operatorname*{id}\rightarrow0$ in $C_{c}^{1}\left(
\mathbb{R}^{n};\mathbb{R}^{n}\right)  $. It suffices to take $E_{j}%
:=r_{j}\left(  E\right)  $.

Let $u_{j}:=1-2\chi_{E_{j}}$. Since $\operatorname{P}_{\Phi}\left(
E_{j},\Omega\right)  =\operatorname{P}_{\Phi}\left(  E_{j}\right)  $, we have
$\mathcal{E}_{0}\left(  u_{j},\Omega\right)  =\operatorname{P}_{\Phi}\left(
E_{j}\right)  $. By \cite[Theorem 3.5(ii)]{bouchitte} for every $j$ there
exists a sequence $\left\{  u_{j}^{k}\right\}  $ converging to $u_{j}$ in
$L^{1}\left(  \Omega\right)  $ such that
\begin{equation}
\operatorname{P}_{\Phi}\left(  E_{j}\right)  =\mathcal{E}_{0}\left(
u_{j},\Omega\right)  =\limsup_{k\rightarrow\infty}\mathcal{E}_{\varepsilon_{k}%
}\left(  u_{j}^{k},\Omega\right)  \quad\text{and\quad}\int_{\Omega}u_{j}%
^{k}\left(  x\right)  \,dx=m\,. \label{bou 2}%
\end{equation}
Since $E_{j}\subset\subset\Omega$, the construction used in \cite[Theorem
3.5(ii)]{bouchitte} allows us to deduce that $u_{j}^{k}\in H^{1}\left(  \Omega\right)  $
and that it is possible to assume $u_{j}^{k}=1$ on $\partial\Omega$. Hence,
$\mathcal{E}_{\varepsilon_{k}}\left(  u_{j}^{k},\Omega\right)
=\mathcal{F}_{\varepsilon_{k}}\left(  u_{j}^{k}\right)  $, so that
(\ref{bou 2}) gives
\[
\operatorname{P}_{\Phi}\left(  E_{j}\right)  \geq\limsup_{k\rightarrow\infty
}\mathcal{F}_{\varepsilon_{k}}\left(  u_{j}^{k}\right)  \geq\mathcal{F}%
_{0}^{\prime\prime}\left(  u_{j}\right)  \,.
\]
Letting $j\rightarrow\infty$ and using the lower semicontinuity of
$\mathcal{F}_{0}^{\prime\prime}$ and the fact that $\operatorname{P}_{\Phi
}\left(  E_{j}\right)  \rightarrow\operatorname{P}_{\Phi}\left(  E\right)  $,
we obtain $\mathcal{F}_{0}\left(  u\right)  =\operatorname{P}_{\Phi}\left(
E\right)  \geq\mathcal{F}_{0}^{\prime\prime}\left(  u\right)  $, which shows
that $\mathcal{F}_{0}\geq\mathcal{F}_{0}^{\prime\prime}$.
\end{proof}

\section{The Liminf Inequality}

\label{section liminf}

By (\ref{functional F zero}) and (\ref{minimal perimeter value}), the $\Gamma
$-liminf inequality (\ref{second gamma liminf}) is a consequence of the
following theorem.

\begin{theorem}
\label{theorem gamma liminf} Let $\{u_{\varepsilon}\} $ be a sequence of
functions satisfying (\ref{constraint}) and converging strongly in $L^{1}(\Omega)$ to a
minimizer $u_{0}=1-2\chi_{E_{0}}$ of $\mathcal{F}_{0}$. Then
\begin{equation}
\liminf_{\varepsilon\rightarrow0+} \frac{\mathcal{E}_{\varepsilon}\left(
u_{\varepsilon}\right)  - n\kappa_{\Phi} c_{W}r^{n-1}}{\varepsilon} \geq0\,.
\label{Gamma liminf ineq}%
\end{equation}

\end{theorem}

\begin{proof} We begin by giving an outline of the proof. The first step is to replace 
$u_\varepsilon$ by a minimizer $\tilde{u}_\varepsilon$ 
 of an auxiliary energy where we relax the mass constraint in \eqref{constraint} 
 with an integral inequality.  The advantage in doing this is that we can use a truncation argument 
to prove that  $\tilde{u}_\varepsilon\le 1$ in $\Omega$. This allows us to use a convex symmetrization 
argument to reduce the energy by replacing $\tilde{u}_{\varepsilon}$ 
with a \textquotedblleft radial\textquotedblright\ 
 function $\hat{w}_\varepsilon$, i.e.,  a function of the form 
$\hat{w}_{\varepsilon}\left(  x\right)  =\overline{w}_{\varepsilon}\left(
\Phi^{\circ}\left(  x\right)  \right) $  defined on the ball $B^{\Phi^{\circ}}_{R}\!\!\left(  0 \right)$
 with the same volume as
$\Omega$. 

   To be precise, $\hat{w}_\varepsilon$ is defined as a ``radial'' minimizer of a problem  in $B^{\Phi^{\circ}}_{R}\!\!\left(  0\right)$
with a suitable  inequality constraint on the mass. 
The one-dimensional function $\overline{w}_{\varepsilon}$ satisfies an Euler--Lagrange equation with a Lagrange multiplier
$\lambda_\varepsilon$ such that $\varepsilon\lambda_\varepsilon\to 0$ as $\varepsilon\to 0+$. 
The choice of the inequality constraint allows us to prove that $\lambda_\varepsilon\ge 0$, which will be important in what follows. 
 
To estimate the energy of  $\overline{w}_{\varepsilon}$  it is convenient to consider the change of variables 
$\rho=r+\varepsilon t$, where $r$ is defined in \eqref{r}, and to introduce the function
$w_{\varepsilon}\left(  t\right)  :=\overline{w}_{\varepsilon}(r+\varepsilon
t)$ for $ -\frac{r}{\varepsilon}\le t\le \frac{R-r}{\varepsilon} $.  
Now the context of our problem has been reduced to a simpler one-dimensional  setting. Indeed, it turns out that   to prove \eqref{Gamma liminf ineq} it is enough to show that
\begin{equation}
\liminf_{\varepsilon\rightarrow0+} \frac{\mathcal{H}_{\varepsilon
}(w_{\varepsilon})- c_{W}r^{n-1}}{\varepsilon} \geq0\,,
\label{Gamma liminf ineq 2}%
\end{equation}
where the functional
\begin{equation}
\mathcal{H}_{\varepsilon}\left(  w\right)  :=\int_{-\frac{r}{\varepsilon}%
}^{\frac{R-r}{\varepsilon}}\left(  W\left(  w\left(  t\right)  \right)
+\left\vert w^{\prime}\left(  t\right)  \right\vert ^{2}\right)  \left(
r+\varepsilon t\right)  ^{n-1} dt \label{Gepsilon}%
\end{equation}
does not contain singular terms in $\varepsilon$. 

The proof of \eqref{Gamma liminf ineq 2} is based on several delicate estimates on $w_\varepsilon$. 
We first show that $w_\varepsilon$ vanishes at a point $\delta_\varepsilon$, with 
$\varepsilon\delta_\varepsilon\to 0$ as $\varepsilon\to 0+$, and introduce the shifted function  
$
\check{w}_{\varepsilon}\left(  t\right)  :=w_{\varepsilon}\left(  t+\delta
_{\varepsilon}\right)$. 
Then we prove that  $\check{w}_{\varepsilon}\rightarrow z$ in $H^{1}_{\text{loc}}$, 
 where $z$ is the  ``optimal  profile'' introduced in \eqref{Cauchy problem}, and that  $\lambda_\varepsilon\to (n-1)c_W$  as $\varepsilon\to 0+$. 
Next  we derive some technical 
estimates on $\check{w}_{\varepsilon}$ using arguments from the theory of ordinary differential equations, which rely on the fact that $\check{w}_{\varepsilon}(0)=0$.  These estimates allow us to show that $\liminf_\varepsilon \delta_{\varepsilon}\ge0$
and to finally prove   \eqref{Gamma liminf ineq 2}.

We divide the proof into a series of steps.

\medskip
\noindent
{\bf Step 1.} {\it Here we replace $u_\varepsilon$ by a minimizer $\tilde{u}_\varepsilon$ 
 of an auxiliary energy where we relax the mass constraint in \eqref{constraint} 
 with an integral inequality. }
 
 To be precise, we introduce the functional 
 $\tilde{\mathcal{F}}_\varepsilon$
defined by
\[
\tilde{\mathcal{F}}_{\varepsilon}\left(  u\right)  :=\mathcal{E}_{\varepsilon}\left(  u\right)
\]
if
\begin{equation}
u\in H^{1}\left(  \Omega\right)  \,,\quad\int_{\Omega}u\left(  x\right)
\,dx\leq m\,,\quad\text{and\quad}u=1\text{ on }\partial\Omega\,.
\label{constraint inequality}%
\end{equation}
We extend $\tilde{\mathcal{F}}_{\varepsilon}$ to $L^{1}\left(  \Omega\right)
$ by setting $\tilde{\mathcal{F}}_{\varepsilon}\left(  u\right)  :=+\infty$ if
(\ref{constraint inequality}) is not satisfied. Then, reasoning as in Theorem
\ref{theorem gamma convergence}, we can show that the $\Gamma$-limit of 
$\{\tilde{\mathcal{F}}_{\varepsilon}\}$ 
is the
functional $\tilde{\mathcal{F}}_{0}$ defined by%
\[
\tilde{\mathcal{F}}_{0}\left(  u\right)  :=c_{W}\operatorname{P}_{\Phi}\left(
E\right)
\]
if
\begin{equation}
u=1-2\chi_{E}\,,\quad E\subset\Omega\,,\quad\operatorname{P}\left(  E\right)
<+\infty\,,\,\text{ and}\quad\int_{\Omega}u\left(  x\right)  \,dx\leq m\,\,,
\label{constraint inequality 2}%
\end{equation}
while $\tilde{\mathcal{F}}_{0}\left(  u\right)  :=+\infty$ if
(\ref{constraint inequality 2}) is not satisfied.

Let $\tilde{u}_{\varepsilon}$ be a minimizer of $\tilde{\mathcal{F}%
}_{\varepsilon}$, whose existence can be justified by the Direct Method of the Calculus of Variations. Then
\begin{equation}
\tilde{\mathcal{F}}_{\varepsilon}\left(  \tilde{u}_{\varepsilon}\right)
\leq\mathcal{E}_{\varepsilon}\left(  u_{\varepsilon}\right)  \,.
\label{u tilde u}%
\end{equation}
Note that, by standard properties of $\Gamma$-convergence, we have that the sequence
\begin{equation}\label{F epsilon bounded}
\{\tilde{\mathcal{F}}_{\varepsilon}(\tilde{u}_\varepsilon)\}\quad\text{is bounded}
\end{equation}
and
 $\left\{  \tilde
{u}_{\varepsilon}\right\}  $ converges strongly in $L^{1}\left(  \Omega\right)  $ to
$u_{0}=1-2\chi_{E_{0}}$, where $E_{0}$ satisfies 
$\operatorname{P}_{\Phi}(E_0)\le \operatorname{P}_{\Phi}(E)$ for
every set $E\subset\Omega$ with finite perimeter and such that
\[
\frac{\left\vert \Omega\right\vert -m}{2}\leq\left\vert E\right\vert \,.
\]%

We claim that
\begin{equation}
\frac{\left\vert \Omega\right\vert -m}{2}=\left\vert E_{0}\right\vert \,.
\label{exact mass}%
\end{equation}
Since the ball $B_{r}^{\Phi^{\circ}}\!\!\left(  y_{0}\right) $ introduced in \eqref{exists} is contained in 
$\Omega$ and satisfies $\frac{\left\vert \Omega\right\vert -m}{2}=\left\vert B_{r}^{\Phi^{\circ}}\!\!\left(  y_{0}\right)\right\vert$, 
we have
\begin{equation}\label{per 0}
\operatorname{P}_{\Phi}\left(  E_{0}\right)\le \operatorname{P}_{\Phi}\left(  B_{r}^{\Phi^{\circ}}\!\!\left(  y_{0}\right)\right)\,.
\end{equation} 
Let $\rho\ge r$ be such that $\left\vert E_0\right\vert=\left\vert B_{\rho}^{\Phi^{\circ}}\!\!\left(  y_{0}\right)\right\vert$. Then for every $F\subset 
\mathbb{R}^n$ with finite perimeter and with $\left\vert F\right\vert=\left\vert E_0\right\vert$, by the minimality of the Wulff shape in $\mathbb{R}^n$ (see \cite{Fon}) it follows
\begin{equation}\label{per 1}
\operatorname{P}_{\Phi}\left(  B_{\rho}^{\Phi^{\circ}}\!\!\left( y_{0}\right)\right)\le \operatorname{P}_{\Phi}\left( F\right)\,.
\end{equation} 
From \eqref{perimeter rho}, \eqref{per 0}, and \eqref{per 1} we obtain $\operatorname{P}_{\Phi}\left(  E_{0}\right)\le \operatorname{P}_{\Phi}\left( F\right)$ for every $F\subset 
\mathbb{R}^n$ with finite perimeter and with $\left\vert F\right\vert=\left\vert E_0\right\vert$.  Since the Wulff sets are the unique minimizers of $\operatorname{P}_{\Phi}$ in  $\mathbb{R}^n$ under the  volume constraint  (see \cite{FonMu}), there exists $x_0\in\Omega$ such that
$E_0=B_{\rho}^{\Phi^{\circ}}\!\!\left(  x_0\right)$.  By \eqref{perimeter rho} and \eqref{per 0} it follows that $\rho=r$ and that \eqref{exact mass} holds.

Next we prove that $\tilde{u}_{\varepsilon}\leq1$ a.e.\ in $\Omega$. Let
$u_{1}:=\min\left\{  \tilde{u}_{\varepsilon},1\right\}  $. Assume, by
contradiction, that $\left\vert \{\tilde{u}_{\varepsilon}>1\}\right\vert >0$.
Since $W\left(  1 \right)  =0$ and $W\left(  s \right)  >0$ for $s>1$ by \eqref{W prime} and  \eqref{W beta}, we have
that $W\left(  u_{1}\left(  x\right)  \right)  \leq W\left(  \tilde
{u}_{\varepsilon}\left(  x\right)  \right)  $ for a.e.\ $x\in\Omega$, and the
inequality is strict for a.e.\ $x\in\{\tilde{u}_{\varepsilon}>1\}$. This
implies that
\[
\int_{\Omega}W\left(  u_{1}\left(  x\right)  \right)  \,dx< \int_{\Omega}
W\left(  \tilde{u}_{\varepsilon}\left(  x\right)  \right)  \,dx\,.
\]
Since $\nabla u_{1}=\nabla\tilde{u}_{\varepsilon}$ a.e. on $\{\tilde
{u}_{\varepsilon}\leq1\}$ and $\nabla u_{1}=0$ a.e. on $\{\tilde
{u}_{\varepsilon}> 1\}$, we have also $\Phi\left(  \nabla u_{1}\left(
x\right)  \right)  \leq\Phi\left(  \nabla\tilde{u}_{\varepsilon}\left(
x\right)  \right)  $ for a.e.\ $x\in\Omega$, which implies
\[
\int_{\Omega}\Phi^{2}\left(  \nabla u_{1}\left(  x\right)  \right)
\,dx\leq\int_{\Omega} \Phi^{2}\left(  \nabla\tilde{u}_{\varepsilon}\left(
x\right)  \right)  \,dx\,.
\]
Noting that $u_{1}$ satisfies (\ref{constraint inequality}), the previous
inequalities give $\tilde{\mathcal{F}}_{\varepsilon}\left(  u_{1}\right)  <
\tilde{\mathcal{F}}_{\varepsilon}\left(  \tilde{u}_{\varepsilon}\right)  $,
which contradicts the minimality of $\tilde{u}_{\varepsilon}$. This proves
that $\tilde{u}_{\varepsilon}\leq1$ a.e.\ in $\Omega$.

\medskip
\noindent {\bf Step 2.} {\it In this step we use 
\textquotedblleft convex\textquotedblright\ rearrangements
 to replace $\tilde{u}_{\varepsilon}$ by a \textquotedblleft radial\textquotedblright\ 
 function $\hat{w}_\varepsilon$, i.e., a function depending on $x$ only through $\Phi^\circ(x)$.
 }
 
Define $\tilde{v}_{\varepsilon}:=1-\tilde{u}_{\varepsilon}\in H_{0}^{1}\left(
\Omega\right)  $ and observe that
\begin{equation}
\mathcal{E}_{\varepsilon}\left(  \tilde{u}_{\varepsilon}\right)  =\int%
_{\Omega}\left(  \frac{1}{\varepsilon}W\left(  1-\tilde{v}_{\varepsilon
}(x)\right)  +\varepsilon\Phi^{2}\left(  \nabla\tilde{v}_{\varepsilon}\left(
x\right)  \right)  \right)  \,dx\,. 
\nonumber
\end{equation}
Since $\tilde{v}_{\varepsilon}\geq0$, we define the \textquotedblleft
convex\textquotedblright\ rearrangement $v_{\varepsilon}^{\star}$ of
$\tilde{v}_{\varepsilon}$ as the unique function of the form
\begin{equation}
v_{\varepsilon}^{\star}\left(  x\right)  =\overline{v}_{\varepsilon}\left(
\Phi^{\circ}\left(  x\right)  \right)  \,, \label{convex rearrangement}%
\end{equation}
with $\overline{v}_{\varepsilon}:\mathbb{R}^{+}\rightarrow\mathbb{R}^{+}$
nonincreasing and continuous from the right, and such that $\left\vert \{\tilde
{v}_{\varepsilon}>t\}\right\vert =\left\vert \{v_{\varepsilon}^{\star
}>t\}\right\vert $ for every $t>0$. It can be shown that 
\begin{equation}
\int_{B_{R}^{\Phi^{\circ}}\!\left(  0\right)  }W\left(  1-v_{\varepsilon
}^{\star}(x)\right)  \,dx=\int_{\Omega}W\left(  1-\tilde{v}_{\varepsilon
}(x)\right)  \,dx\,,\nonumber
\end{equation}
where $\left\vert B_{R}^{\Phi^{\circ}}\!\!\left(  0\right)  \right\vert
=\left\vert \Omega\right\vert $. The P\'{o}lya-Szeg\"{o} principle, which holds
also for \textquotedblleft convex\textquotedblright\ rearrangements (see
\cite[Theorem 3.1]{AFTL}), gives  $v_{\varepsilon}^{\star}\in H_{0}%
^{1}\left(  B_{R}^{\Phi^{\circ}}\!\!\left(  0\right)  \right)  $ and
\begin{equation}
\int_{B_{R}^{\Phi^{\circ}}\!\left(  0\right)  }\Phi^{2}\left(  \nabla
v_{\varepsilon}^{\star}(x)\right)  \,dx\leq\int_{\Omega}\Phi^{2}\left(
\nabla\tilde{v}_{\varepsilon}(x)\right)  \,dx\,.\nonumber
\end{equation}%
Therefore, we deduce that%
\begin{equation}
\mathcal{E}_{\varepsilon}(1-v_{\varepsilon}^{\star},B_{R}^{\Phi^{\circ}}\!\left(  0\right))
\leq \tilde{\mathcal{F}}%
_{\varepsilon}\left(  \tilde{u}_{\varepsilon}\right)  \,.
\label{spherical rearrangement}%
\end{equation}

Let $\hat{w}_{\varepsilon}$ be a minimizer of 
$\mathcal{E}_{\varepsilon}(\cdot,B_{R}^{\Phi^{\circ}}\!\left(  0\right))$
among all functions $u\in H^{1}\left(  B_{R}^{\Phi^{\circ}}\!\left(  0\right)  \right)$ 
satisfying
\begin{equation}
\nonumber
\int_{B_{R}^{\Phi^{\circ}}\!\left(  0\right)  }u\left(  x\right)
\,dx\leq m\quad\text{and\quad}u=1\text{ on }\partial B_{R}^{\Phi^{\circ}%
}\!\left(  0\right)  \,.
\end{equation}
By \eqref{spherical rearrangement} we have
\begin{equation}
 \mathcal{E}_{\varepsilon}(\hat{w}_{\varepsilon},B_{R}^{\Phi^{\circ}}\!\left(  0\right))
\leq 
\mathcal{E}_{\varepsilon}(1-v_{\varepsilon}^{\star},B_{R}^{\Phi^{\circ}}\!\left(  0\right))
\leq \tilde{\mathcal{F}}%
_{\varepsilon}\left(  \tilde{u}_{\varepsilon}\right)  \,.
\label{spherical rearrangement3}%
\end{equation}
 
 Reasoning as in Step 1, with $\tilde{\mathcal{F}}_{\varepsilon}$ replaced by
$\mathcal{E}_{\varepsilon}(\cdot,B_{R}^{\Phi^{\circ}}\!\left(  0\right))$, we may assume that
\begin{equation}
\hat{w}_{\varepsilon}\rightarrow1-2\chi_{B_{r}%
^{\Phi^{\circ}}\!\left(  0\right)  }\quad\text{in }L^{1}\left(  B_{R}%
^{\Phi^{\circ}}\!\left(  0\right)  \right)  \,, \label{convergence in L1}%
\end{equation}
where $r$ is given by (\ref{r}).

Using a symmetrization argument similar to the one above 
(see \eqref{convex rearrangement} and \eqref{spherical rearrangement}), we can assume that there exists a function $\overline{w}_{\varepsilon}:\mathbb{R}^{+}\rightarrow\mathbb{R}^{+}$,
nondecreasing and continuous from the right, such that
\begin{equation}
\hat{w}_{\varepsilon}\left(  x\right)  =\overline{w}_{\varepsilon}\left(
\Phi^{\circ}\left(  x\right)  \right)  \,. \label{convex rearrangement2}%
\end{equation}
By (\ref{derivative polar}) and (\ref{convex rearrangement}) we have
$\Phi\left(  \nabla \hat{w}_{\varepsilon}\left(  x\right)  \right)
=\overline{w}_{\varepsilon}^{\prime}\left(  \Phi^{\circ}\left(  x\right)
\right)  $ for a.e. $x\in B_{R}^{\Phi^{\circ}}\!\!\left(  0\right)  $, and so
(\ref{integration polar}) yields
\begin{equation}\label{polar coordinates}
\mathcal{E}_{\varepsilon}(\hat{w}_{\varepsilon},B_{R}^{\Phi^{\circ}}\!\left(  0\right))=n\kappa_{\Phi}\int%
_{0}^{R}\left(  \frac{1}{\varepsilon}W\left( \overline{w}_{\varepsilon
}(\rho)\right)  +\varepsilon\left\vert \overline{w}_{\varepsilon}^{\prime
}\left(  \rho\right)  \right\vert ^{2}\right)  \rho^{n-1}d\rho\,.
\end{equation}%
\color{black}%

\medskip 
\noindent {\bf Step 3. } {\it In this step we prove some elementary properties 
of the one-dimen\-sional function $\overline{w}_\varepsilon$ introduced 
in \eqref{convex rearrangement2}. }

By the minimality of $\hat{w}_{\varepsilon}$,  restricting our attention to functions of the form $w\circ\Phi^\circ$  
 and using \eqref{integration polar} and \eqref{derivative polar},  we deduce that $\overline{w}_{\varepsilon}$ is a minimizer of  
\begin{equation}
\mathcal{G}_{\varepsilon}\left(  w\right)  :=\int_{0}^{R}\left(  \frac
{1}{\varepsilon}W\left(  w(\rho)\right)  +\varepsilon\left\vert w^{\prime
}\left(  \rho\right)  \right\vert ^{2}\right)  \rho^{n-1}d\rho\,
\label{G epsilon} 
\end{equation}
among all functions $w\in H_{\operatorname{loc}}^{1}\left(  0,R\right)  $ subject
to the constraints
\begin{equation}
\int_{0}^{R}\left\vert w^{\prime}(\rho)\right\vert ^{2}\rho^{n-1}d\rho
<+\infty,\quad w\left(  R\right)  =1,\quad n\kappa_{\Phi}\int_{0}^{R}w\left(
\rho\right)  \rho^{n-1}d\rho\leq m\,. \label{radial constraint}%
\end{equation}

Note that if $w$ satisfies the first two conditions in (\ref{radial constraint}) then
\[
\left\vert w\left(  \rho\right)  \right\vert \leq1+\int_{\rho}^{R}\left|w^{\prime
}\left(  \sigma\right)\right|  d\sigma\,,
\]
hence
\[
\left\vert w\left(  \rho\right)  \rho^{n-1}\right\vert \leq\rho^{n-1}%
+\int_{\rho}^{R}\left|w^{\prime
}\left(  \sigma\right)\right|  \sigma^{n-1}d\sigma\,.
\]
By H\"{o}lder's inequality we obtain
\[
\left\vert w\left(  \rho\right)  \rho^{n-1}\right\vert \leq R^{n-1}+\left(
\frac{R^{n}}{n}\right)  ^{1/2}\left(  \int_{0}^{R}\left|w^{\prime
}\left(  \sigma\right)\right|^2 \sigma^{n-1}d\sigma\right)  ^{1/2}\,,
\]
hence
\begin{equation}
w(\rho)\rho^{n-1}\quad\text{is bounded in }(0,R)\,. \label{w rho n-1}%
\end{equation}
This implies that the integral in the last inequality in
(\ref{radial constraint}) is well defined. 

Reasoning by truncation as at the end of Step 1, we can prove
that 
\begin{equation}\label{less than one}
\overline{w}_{\varepsilon}\left(  \rho\right)  \leq1\quad\text {for  } 0<\rho<R  \,.
\end{equation}

 Using the equalities 
$\hat{w}_{\varepsilon}=\overline{w}_{\varepsilon}\circ
\Phi^{\circ}$ and 
$\chi_{B_{r}%
^{\Phi^{\circ}}\!\left(  0\right)  }=\chi_{[0,r)}\circ\Phi^\circ$, by \eqref{convergence in L1}
we obtain that 
\begin{equation}
\nonumber
\overline{w}_\varepsilon\rightarrow 1-2\chi_{[0,r)} \quad\text{in }L^{1}\left(0,R\right)
\,.
\end{equation}
Since each $\overline{w}_{\varepsilon}$ is nondecreasing, it follows that there is pointwise
convergence at every $\rho$, with the possible exceptions of $0$ and $r$. In particular, we have
\begin{equation}
\overline{w}_{\varepsilon}\left(  \rho\right)  \rightarrow\left\{
\begin{array}
[c]{ll}%
-1 & \text{if }0<\rho<r\,,\\
\phantom{-} 1 & \text{if }r<\rho\le R\,,
\end{array}
\right.  \label{convergence}%
\end{equation}
where the case $\rho=R$ can be obtained by \eqref{less than one}.
 In turn, for every $0<\rho_0<R$ the sequence 
\begin{equation}\label{w epsilon bounded}
\{\overline{w}_{\varepsilon}\}\quad\text{is bounded in }L^\infty([\rho_0,R])\,.
\end{equation}

\medskip
\noindent {\bf Step 4.} {\it Here we derive  the Euler--Lagrange
equation for $\overline{w}_{\varepsilon}$, and we prove that the corresponding Lagrange multipliers
 $\lambda_\varepsilon$ satisfy}
 \begin{equation}\label{epsilon lambda go zero}
\lim_{\varepsilon\to 0+}\varepsilon\lambda_\varepsilon= 0\,.
\end{equation}

We claim that $\overline{w}_{\varepsilon}\in C^{2}\left(  0,R\right)$ and satisfies the Euler--Lagrange
equation
\begin{equation}
-2\varepsilon\overline{w}_{\varepsilon}^{\prime\prime}\left(  \rho\right)
-\frac{2\left(  n-1\right)  \varepsilon}{\rho}\overline{w}_{\varepsilon
}^{\prime}\left(  \rho\right)  +\frac{1}{\varepsilon}W^{\prime}\left(
\overline{w}_{\varepsilon}(\rho)\right)  =-\lambda_{\varepsilon}\leq0
\label{radial Lagrange multiplier}%
\end{equation}
for some constant $\lambda_{\varepsilon}\geq0$. To see this, let
$\varphi\in C_{c}^{\infty}\left(  0,R\right)  $ be nonnegative. For $t>0$
the function $\overline{w}_{\varepsilon}-t\varphi$ fulfills
(\ref{radial constraint}), and so%
\[
\mathcal{G}_{\varepsilon}\left(  \overline{w}_{\varepsilon}\right)
\leq\mathcal{G}_{\varepsilon}\left(  \overline{w}_{\varepsilon
}-t\varphi\right)  \,.
\]
Thus the derivative of the function $t\mapsto\mathcal{G}_{\varepsilon
}\left(  \overline{w}_{\varepsilon}-t\varphi\right)  $ is greater than or equal
to $0$ at $t=0$. This gives%
\begin{equation} \label{less}
\int_{0}^{R}\left(  \frac{1}{\varepsilon}W^{\prime}\left(  \overline
{w}_{\varepsilon}(\rho)\right)  \varphi\left(  \rho\right)  +2\varepsilon
\,\overline{w}_{\varepsilon}^{\prime}\left(  \rho\right)  \varphi^{\prime
}\left(  \rho\right)  \right)  \rho^{n-1}\,d\rho\leq0
\end{equation}
for all nonnegative $\varphi\in C_{c}^{\infty}\left(  0,R\right)  $, which
shows that $-2\varepsilon\left(  \overline{w}_{\varepsilon}^{\prime}\left(
\rho\right)  \rho^{n-1}\right)  ^{\prime}+\frac{1}{\varepsilon}W^{\prime
}\left(  \overline{w}_{\varepsilon}(\rho)\right)  \rho^{n-1}$ is nonpositive in the sense of 
distributions. 

On the other hand, if we consider $\psi\in C_{c}^{\infty
}\left(  0,R\right)  $ such that $\int_{0}^{R}\psi\left(  \rho\right)
\rho^{n-1}d\rho=0$, then $\overline{w}_{\varepsilon}+t\psi$ satisfies
(\ref{radial constraint}) for all $t\in\mathbb{R}$, and so
\begin{equation}\label{equal}
\int_{0}^{R}\left(  \frac{1}{\varepsilon}W^{\prime}\left(  \overline
{w}_{\varepsilon}(\rho)\right)  \psi\left(  \rho\right)  +2\varepsilon
\,\overline{w}_{\varepsilon}^{\prime}\left(  \rho\right)  \psi^{\prime
}\left(  \rho\right)  \right)  \rho^{n-1}\,d\rho=0
\end{equation}
for all $\psi\in C_{c}^{\infty}\left(  0,R\right)  $ with $\int_{0}^{R}\psi\left(  \rho\right)  \rho^{n-1}\,d\rho=0$. 
We now use a classical argument (see, e.g., \cite[Lemma 7.3]{leoni}) to show that \eqref{equal} implies that there exists a  constant $\lambda_{\varepsilon}$ such that
\begin{equation}
-2\varepsilon\left(  \overline{w}_{\varepsilon}^{\prime}\left(  \rho\right)
\rho^{n-1}\right)  ^{\prime}+\frac{1}{\varepsilon}W^{\prime}\left(
\overline{w}_{\varepsilon}\left(  \rho\right)  \right)  \rho^{n-1}%
=-\lambda_{\varepsilon}\rho^{n-1}\label{Euler second form}%
\end{equation}
in the sense of distributions in $(0,R)$.  

Fix $\varphi_1\in C_{c}^{\infty}\left(  0,R\right)  $ with $\int_{0}%
^{R}\varphi_1\left(  \rho\right)  \rho^{n-1}\,d\rho=1$.
Given $\varphi\in C_{c}^{\infty
}\left(  0,R\right)  $, we can write $
\varphi=c_\varphi \varphi_1+\psi$, 
where $
c_\varphi :=\int_{0}^{R}\varphi\left(  \rho\right)
\rho^{n-1}\,d\rho$ 
and $\psi\in C_{c}^{\infty
}\left(  0,R\right) $ satisfies $\int_{0}^{R}\psi\left(  \rho\right)  \rho^{n-1}\,d\rho=0$.
 Hence, using \eqref{equal}, we obtain
\begin{equation*}
\int_{0}^{R}\left(  \frac{1}{\varepsilon}W^{\prime}\left(  \overline
{w}_{\varepsilon}(\rho)\right)  \varphi\left(  \rho\right)  +2\varepsilon
\,\overline{w}_{\varepsilon}^{\prime}\left(  \rho\right)  \varphi^{\prime
}\left(  \rho\right)  \right)  \rho^{n-1}\,d\rho=-\lambda_\varepsilon\int_{0}^{R}\varphi\left(  \rho\right)
\rho^{n-1}\,d\rho\,,
\end{equation*}
where
\begin{equation}\label{lambda epsilon}
\lambda_\varepsilon:=-
\int_{0}^{R}\left(  \frac{1}{\varepsilon}W^{\prime}\left(  \overline
{w}_{\varepsilon}(\rho)\right)  \varphi_1\left(  \rho\right)  +2\varepsilon
\,\overline{w}_{\varepsilon}^{\prime}\left(  \rho\right)  \varphi_1^{\prime
}\left(  \rho\right)  \right)  \rho^{n-1}\,d\rho\,.
\end{equation}
This concludes the proof of \eqref{Euler second form}. 

By \eqref{less} it follows that $\lambda_\varepsilon\ge 0$. Using the facts that $W$ is of class $C^1$ and  that $\overline{w}_\varepsilon$ is bounded on $[\rho_0,R]$ for every $0<\rho_0<R$ by \eqref{w rho n-1}, we deduce that $\overline{w}_{\varepsilon}\in C^{2}\left(  0,R\right)$ and that 
 (\ref{radial Lagrange multiplier}) and \eqref{Euler second form} are satisfied pointwise. 

Next we prove \eqref{epsilon lambda go zero}. By \eqref{lambda epsilon},
\begin{equation}
\nonumber
\varepsilon\lambda_\varepsilon=-
\int_{0}^{R}W^{\prime}\left(  \overline
{w}_{\varepsilon}(\rho)\right)  \varphi_1\left(  \rho\right)   \rho^{n-1}\,d\rho
-\int_{0}^{R} 2\varepsilon^2
\,\overline{w}_{\varepsilon}^{\prime}\left(  \rho\right)  \varphi_1^{\prime
}\left(  \rho\right)   \rho^{n-1}\,d\rho\,.
\end{equation}
Since $\varphi_1$ has compact support in $(0,R)$, the first integral tends to zero 
 in view of \eqref{W even}, 
 \eqref{W beta}, \eqref{convergence}, and \eqref{w epsilon bounded}, 
while the second integral goes to zero since the sequence $\{\varepsilon |\overline{w}_\varepsilon^\prime|^2 \rho^{n-1}\}$ is bounded 
in $L^1(0,R)$ by \eqref{F epsilon bounded}, \eqref{spherical rearrangement3}, and \eqref{polar coordinates}.  This concludes the proof of \eqref{epsilon lambda go zero}.

\medskip
\noindent {\bf Step 5.} {\it Here we prove that for $\varepsilon>0$ small enough
\begin{equation}
\inf\overline{w}_{\varepsilon}\geq-1-\left(  \frac{\varepsilon\lambda
_{\varepsilon}}{\beta}\right)  ^{1/\left(  \beta-1\right)  }\,,
\label{inf w pound}%
\end{equation}
 where $1<\beta<2$ is the constant in \eqref{W beta}.}

Integrating (\ref{Euler second form}) gives 
\begin{align}
&  \overline{w}^{\prime}_{\varepsilon}\left(  \rho_{2} \right)  \rho_{2}^{n-1}
- \overline{w}^{\prime}_{\varepsilon}\left(  \rho_{1} \right)  \rho_{1}%
^{n-1}\nonumber\\
&  = \frac{1}{2\varepsilon} \int_{\rho_{1}}^{\rho_{2}}
\left(  \frac{1}{\varepsilon} W^{\prime}\left(  \overline{w}_{\varepsilon}\left(  \rho\right)
\right)  \rho^{n-1} + \lambda_{\varepsilon}\rho^{n-1} \right)  d\rho
\label{integrated Euler}%
\end{align}
for $0<\rho_{1}\leq\rho_{2}\leq R$. 
Since $\overline{w}_{\varepsilon}\left(  \rho\right)\le 1$ for every $0<\rho\le R$ by \eqref{less than one}, 
and $W^{\prime}\left(  s \right) <0$
for $s<-1$ by \eqref{W even} and \eqref{W prime}, the integral
\[
\int_{0}^{R}  \left(  \frac{1}{\varepsilon}W^{\prime}\left(  \overline
{w}_{\varepsilon}\left(  \rho\right)  \right)  \rho^{n-1} + \lambda
_{\varepsilon}\rho^{n-1} \right)  d\rho
\]
is well-defined as an element of $\mathbb{R}\cup\left\{  -\infty\right\}  $.
We claim that
\begin{equation}
\lim_{\rho\rightarrow0+} \overline{w}^{\prime}_{\varepsilon}\left(
\rho\right)  \rho^{n-1}= 0\,. \label{w prime rho n-1}%
\end{equation}
First, we observe that the limit 
exists in $\mathbb{R}\cup\left\{  +\infty\right\}  $ by \eqref{integrated Euler}. If it were different from zero, then
there would exist $c_0>0$ and $\rho_{0}>0$ such that $\left\vert \overline{w}^{\prime
}_{\varepsilon}\left(  \rho\right)  \right\vert \geq c_0
/ \rho^{n-1}  $ for $0<\rho<\rho_{0}$. It would follow that
\[
\int_{0}^{R}\left\vert w^{\prime}(\rho) \right\vert ^{2} \rho^{n-1}d\rho
\geq c_0^2\int_{0}^{\rho_{0}} \frac{d\rho
}{\rho^{n-1}}=+\infty\,,
\]
which would contradict the first inequality in (\ref{radial constraint}) since $n\ge 2$. This gives  \eqref{w prime rho n-1}.

To prove \eqref{inf w pound} we first show that 
\begin{equation}\label{lim inf -1}
 \liminf_{\varepsilon\rightarrow0+}\, \inf\overline{w}_{\varepsilon}\geq-1\,.
\end{equation}
It is enough to prove an estimate from below for those $\varepsilon$ such that $\inf\overline{w}_{\varepsilon}<-1$. We claim that for those $\varepsilon$,
\begin{equation}
\limsup_{\rho\rightarrow0+} \left(  W^{\prime}\left(  \overline{w}_{\varepsilon
}\left(  \rho\right)  \right)  + \varepsilon\lambda_{\varepsilon} \right)
\geq0\,. \label{W prime epsilon lambda}%
\end{equation}
If not, (\ref{integrated Euler}) and (\ref{w prime rho n-1}) imply that
$\overline{w}^{\prime}_{\varepsilon}\left(  \rho\right)  \rho^{n-1}<0$ for
$\rho>0$ small enough, and this violates the fact that $\overline
{w}_{\varepsilon}$ is nondecreasing. This proves (\ref{W prime epsilon lambda}%
), which implies that, for $\inf\overline{w}_{\varepsilon}<-1$, we have
\begin{equation}
W^{\prime}\left(  \inf\overline{w}_{\varepsilon} \right)  + \varepsilon
\lambda_{\varepsilon} \geq0\,, \label{W prime leq}%
\end{equation}
 where 
$W^\prime(-\infty)$ denotes the limsup of $W^\prime(s)$ as $s\to -\infty$.

 Since $W^{\prime}(s)<0$ for
every $-\infty\le s<-1$, by \eqref{W even}, \eqref{W prime},  and  \eqref{lim W prime},   inequality  \eqref{lim inf -1} follows from  \eqref{epsilon lambda go zero}. 

By (\ref{W even}), (\ref{W beta}), \eqref{lim inf -1},   and (\ref{W prime leq}), if 
$\varepsilon>0$ is small enough and
$\inf
\overline{w}_{\varepsilon}<-1$ we have
\[
\beta\left(  -\inf\overline{w}_{\varepsilon} -1\right)  ^{\beta-1}%
\leq\varepsilon\lambda_{\varepsilon}\,.
\]
 This proves \eqref{inf w pound}.

\medskip
\noindent {\bf Step 6.} {\it In this step we consider the change of variables 
$\rho=r+\varepsilon t$, we define
$w_{\varepsilon}\left(  t\right)  :=\overline{w}_{\varepsilon}(r+\varepsilon
t)$ for $  -\frac{r}{\varepsilon}\le t\le \frac{R-r}{\varepsilon} $,
and we derive the one-dimensional formulation \eqref{Gamma liminf ineq 2} of the problem in terms of the new energy $\mathcal{H}_\varepsilon$  introduced in \eqref{Gepsilon}.}

Note that $w_{\varepsilon}\left(  \frac{R-r}{\varepsilon}\right)  =1$ and
$w_{\varepsilon}$ is nondecreasing.  By (\ref{inf w pound}), for $\varepsilon>0$ small enough we have 
\begin{equation}
\inf w_{\varepsilon}\geq-1-\left(  \frac{\varepsilon\lambda_{\varepsilon}%
}{\beta}\right)  ^{1/\left(  \beta-1\right)  }\,. 
\nonumber
\end{equation}
In particular, using also \eqref{less than one} and \eqref{epsilon lambda go zero}, for all $\varepsilon>0$ sufficiently small we get
\begin{equation}\label{bound w epsilon}
 -2\le w_{\varepsilon}(t)\le 1
\end{equation}
for all $  -\frac{r}{\varepsilon}\le t\le \frac{R-r}{\varepsilon}$.

Moreover, by \eqref{radial Lagrange multiplier}, $w_{\varepsilon}$ satisfies the Euler--Lagrange equation%
\begin{equation}
-2w_{\varepsilon}^{\prime\prime}\left(  t\right)  -2\left(  n-1 \right)
\varepsilon\frac{w_{\varepsilon}^{\prime}\left(  t\right)  }{r+\varepsilon
t}+W^{\prime}\left(  w_{\varepsilon}\left(  t\right)  \right)  =-\varepsilon
\lambda_{\varepsilon}\leq0\,, \label{e-l}%
\end{equation}
and by (\ref{radial constraint}),
\begin{equation}
\int_{-\frac{r}{\varepsilon}}^{\frac{R-r}{\varepsilon}}w_{\varepsilon}\left(
t\right)  \left(  r+\varepsilon t\right)  ^{n-1}dt\leq\frac{m}{n\kappa_{\Phi
}\varepsilon}\,. \label{volume constraint epsilon}%
\end{equation}

Observe that, setting
\begin{equation}
w_{0}\left(  t\right)  :=\left\{
\begin{array}
[c]{ll}%
-1 & \text{if }t<0\,,
\\\phantom{-} 1 & \text{if }t>0\,,
\end{array}
\right.
\label{w0} 
\end{equation}
we have
\begin{align*}
\int_{-\frac{r}{\varepsilon}}^{\frac{R-r}{\varepsilon}}&w_{0}\left(  t\right)
\left(  r+\varepsilon t\right)  ^{n-1}dt    =\int_{0}^{\frac{R-r}%
{\varepsilon}}\left(  r+\varepsilon t\right)  ^{n-1}dt-\int_{-\frac
{r}{\varepsilon}}^{0}\left(  r+\varepsilon t\right)  ^{n-1}dt\\
&  =\frac{1}{n\varepsilon}\left[  \left(  r+\varepsilon t\right)  ^{n}\right]
_{0}^{\frac{R-r}{\varepsilon}}-\frac{1}{n\varepsilon}\left[  \left(
r+\varepsilon t\right)  ^{n}\right]  _{-\frac{r}{\varepsilon}}^{0}  =\frac{1}{n\varepsilon}\left(  R^{n}-2r^{n}\right)  =\frac{m}{n\kappa
_{\Phi}\varepsilon}\,,
\end{align*}
where the last equality follows from  
 (\ref{constraint reached}), taking into account the fact that $\kappa
_{\Phi} R^n=\left\vert
B^{\Phi^{\circ}}_{R}\!\!\left(  0 \right)  \right\vert =\left\vert
\Omega\right\vert $. Thus, (\ref{volume constraint epsilon}) is equivalent to%
\begin{equation}
\int_{-\frac{r}{\varepsilon}}^{\frac{R-r}{\varepsilon}}\left(  w_{\varepsilon
}\left(  t\right)  -w_{0}\left(  t\right)  \right)  \left(  r+\varepsilon
t\right)  ^{n-1}dt\leq0\,. \label{rescaled volume constraint}%
\end{equation}

The minimality of $\overline{w}_\varepsilon$ for $\mathcal{G}_\varepsilon$ and a change of variables show that $w_{\varepsilon}$ is a minimizer of the functional $\mathcal{H}_{\varepsilon}$ 
defined in \eqref{Gepsilon} over all
$w\in H^1_{\text{loc}}\left(-\frac{r}{\varepsilon},\frac{R-r}{\varepsilon}\right)\cap H^1\left(0,\frac{R-r}{\varepsilon}\right)$ satisfying $w\left(  \frac{R-r}{\varepsilon}\right)  =1$ and%
\begin{equation}
\int_{-\frac{r}{\varepsilon}}^{\frac{R-r}{\varepsilon}}\left(  w\left(
t\right)  -w_{0}\left(  t\right)  \right)  \left(  r+\varepsilon t\right)
^{n-1} dt\leq0\,. \label{constraint less zero}%
\end{equation}

By (\ref{u tilde u}), \eqref{spherical rearrangement},  \eqref{spherical rearrangement3},  (\ref{polar coordinates}), and \eqref{G epsilon} we have%
\begin{equation*}
  \mathcal{H}_{\varepsilon}\left(  w_{\varepsilon}\right)
  =\mathcal{G}_{\varepsilon}\left(  \overline{w}_{\varepsilon}\right)
=\frac{1}{n\kappa_{\Phi}}\mathcal{E}_{\varepsilon}(\hat{w}_{\varepsilon},B_{R}^{\Phi^{\circ}}\!\left(  0\right))
\leq\frac{1}{n\kappa_{\Phi}} \tilde{\mathcal{F}}_{\varepsilon}\left(
\tilde{u}_{\varepsilon}\right)  \leq\frac{1}{n\kappa_{\Phi}} \mathcal{E}%
_{\varepsilon}\left(  u_{\varepsilon}\right)  \,. 
\end{equation*}
Therefore, in order to prove (\ref{Gamma liminf ineq}) it is enough
to show that \eqref{Gamma liminf ineq 2} holds. 

\medskip
\noindent {\bf Step 7.} {\it Here we prove that the function $w_{\varepsilon}$
obtained in the previous step vanishes at a point $\delta_{\varepsilon}$, 
and that}
\begin{equation}
\lim_{\varepsilon\rightarrow0^{+}}\varepsilon\delta_{\varepsilon}=0\,.
\label{epsilon delta epsilon go to zero}%
\end{equation}

Let $z$ be the function defined by (\ref{Cauchy problem}). We claim that $w=z$ satisfies (\ref{constraint less zero}) for all $\varepsilon>0$ sufficiently small.
 Since $z(t)=w_0(t)$ for $|t|\ge \tau_W$, (\ref{constraint less zero}) reduces to
 \begin{equation*}
\int_{-\tau_W}^{\tau_W}\left(  z\left(
t\right)  -w_{0}\left(  t\right)  \right)  \left(  r+\varepsilon t\right)
^{n-1} dt\leq0 
\end{equation*}
for all $\varepsilon>0$ sufficiently small, or, equivalently,
\begin{equation*}
\int_{0}^{\tau_W}\left(  z\left(
t\right)  -w_{0}\left(  t\right)  \right)  \left(  r+\varepsilon t\right)
^{n-1} dt\le 
-\int_{-\tau_W}^{0}\left(  z\left(
t\right)  -w_{0}\left(  t\right)  \right)  \left(  r+\varepsilon t\right)
^{n-1} dt \,.
\end{equation*}
Using the fact that $z-w_0$ is odd, a change of variables on the right-hand side
leads to
\begin{equation*}
\int_{0}^{\tau_W}\left(  z\left(
t\right)  -w_{0}\left(  t\right)  \right)  \left(  r+\varepsilon t\right)
^{n-1} dt\le 
\int_{0}^{\tau_W}\left(  z\left(
t\right)  -w_{0}\left(  t\right)  \right)  \left(  r-\varepsilon t\right)
^{n-1} dt \,,
\end{equation*}
which   follows from the fact that $z\left(  t\right)  -w_{0}\left(
t\right)  \leq0$ for all $ 0\le t\le \tau_W$.
Hence, $z$ satisfies (\ref{constraint less zero}) for all $\varepsilon>0$ sufficiently small and so, by the minimality of $w_\varepsilon$, we have  $\mathcal{H}_{\varepsilon}\left(  w_\varepsilon\right)\le \mathcal{H}_{\varepsilon}\left(  z\right)$.  

Moreover, by \eqref{constant cW},%
\begin{align*}
\mathcal{H}_{\varepsilon}\left(  z\right)   &  =\int_{-\tau_{W}}^{\tau_{W}%
}\left(  W\left(  z\left(  t\right)  \right)  +\left\vert z^{\prime}\left(
t\right)  \right\vert ^{2}\right)  \left(  r+\varepsilon t\right)  ^{n-1}dt\\
&  =\int_{-\tau_{W}}^{\tau_{W}}\left(  W\left(  z\left(  t\right)  \right)
+\left\vert z^{\prime}\left(  t\right)  \right\vert ^{2}\right)
r^{n-1}\,dt+O\left(  \varepsilon^{2}\right)  = c_{W}r^{n-1}+O\left(
\varepsilon^{2}\right)  \,,
\end{align*}
where we have used the fact
\[
\int_{-\tau_{W}}^{\tau_{W}}\left(  W\left(  z\left(  t\right)  \right)
+\left\vert z^{\prime}\left(  t\right)  \right\vert ^{2}\right)  t\,dt=0\,,
\]
since $W$ is even and $z$ is odd. It
follows that
\begin{equation}
\mathcal{H}_{\varepsilon}\left(  w_{\varepsilon}\right)  \leq\mathcal{H}%
_{\varepsilon}\left(  z\right)  =c_{W}r^{n-1}+O\left(  \varepsilon^{2}\right)
\,. \label{bound G epsilon}%
\end{equation}

Fix $0<r_{1}<r<r_{2}<R$. Then $\overline{w}_{\varepsilon}(r_{1})\rightarrow-1$
and $\overline{w}_{\varepsilon}(r_{2})\rightarrow1$ by (\ref{convergence}),
and so $w_{\varepsilon}\left(  \frac{r_{1}-r}{\varepsilon}\right)
\rightarrow-1$ and $w_{\varepsilon}\left(  \frac{r_{2}-r}{\varepsilon}\right)
\rightarrow1$ as $\varepsilon\rightarrow0^{+}$. Since $w_{\varepsilon}$ is
continuous,\ for all $\varepsilon$ sufficiently small there
exists $\delta_{\varepsilon}\in\left(  \frac{r_{1}-r}{\varepsilon},\frac
{r_{2}-r}{\varepsilon}\right)  $ such that
\[
w_{\varepsilon}\left(  \delta_{\varepsilon}\right)  =0\,.
\]
Then $r_{1}-r\leq\varepsilon\delta_{\varepsilon}\leq r_{2}-r$ all
$\varepsilon$ sufficiently small. Hence,
\[
r_{1}-r\leq\liminf_{\varepsilon\rightarrow0^{+}}\varepsilon\delta
_{\varepsilon}\leq\limsup_{\varepsilon\rightarrow0^{+}}\varepsilon
\delta_{\varepsilon}\leq r_{2}-r\,.
\]
Letting $r_{1}\rightarrow r^{-}$ and $r_{2}\rightarrow r^{+}$, we conclude  that
\eqref{epsilon delta epsilon go to zero} holds.

\medskip
\noindent {\bf Step 8.} {\it 
Define
\begin{equation}
\check{w}_{\varepsilon}\left(  t\right)  :=w_{\varepsilon}\left(  t+\delta
_{\varepsilon}\right)  \quad\text{for } -\tfrac{r}{\varepsilon} -
\delta_{\varepsilon}\le t\le\tfrac{R-r}{\varepsilon} - \delta_{\varepsilon}
\,.
\label{w check} 
\end{equation}
In this step we prove that 
\begin{equation}
\label{hat w epsilon to z}\check{w}_{\varepsilon}\rightarrow z \quad
\hbox{strongly in}\quad H^{1}\left(  -b,b\right)
\end{equation}
for every $b>0$, where $z$ is the ``optimal profile'' introduced in \eqref{Cauchy problem}.
 }

Fix  $0<r_{1}<r$, $r_{2}<R-r$, and $b \geq\tau_{W}$, where $\tau_{W}$ is given in \eqref{tau W}. By
(\ref{epsilon delta epsilon go to zero}), for all $\varepsilon>0$ sufficiently
small so that $\left\vert \delta_{\varepsilon}\right\vert <\min\left\{  \frac{r-r_{1}%
}{\varepsilon},\frac{R-r-r_{2}}{\varepsilon}\right\}  $,  by \eqref{G epsilon},  \eqref{epsilon delta epsilon go to zero},   and 
(\ref{bound G epsilon}) we obtain
\begin{align}
c_{W}r^{n-1}  &  +O\left(  \varepsilon^{2}\right)  \geq\mathcal{H}%
_{\varepsilon}\left(  w_{\varepsilon}\right) \nonumber\\
&  =\int_{-\frac{r}{\varepsilon} - \delta_{\varepsilon}}^{\frac{R-r}%
{\varepsilon} - \delta_{\varepsilon}}\left(  W\left(  \check{w}_{\varepsilon
}\left(  s\right)  \right)  +\left\vert \check{w}_{\varepsilon}^{\prime}\left(
s\right)  \right\vert ^{2}\right)  \left(  r+\varepsilon s + \varepsilon
\delta_{\varepsilon}\right)  ^{n-1} ds\nonumber\\
&  \geq\int_{-\frac{r_{1}}{\varepsilon}}^{\frac{r_{2}}{\varepsilon}}\left(
W\left(  \check{w}_{\varepsilon}\left(  s\right)  \right)  +\left\vert \check
{w}_{\varepsilon}^{\prime}\left(  s\right)  \right\vert ^{2}\right)  \left(
r+\varepsilon s + \varepsilon\delta_{\varepsilon}\right)  ^{n-1}
ds\label{estimate 5}\\
&  \geq\int_{-b}^{b}\left(  W\left(  \check{w}_{\varepsilon}\left(  s\right)
\right)  +\left\vert \check{w}_{\varepsilon}^{\prime}\left(  s\right)
\right\vert ^{2}\right)  \left(  r+\varepsilon s + \varepsilon\delta
_{\varepsilon}\right)  ^{n-1} ds\,.\nonumber
\end{align}
Fix $0<\eta<r$. Again by (\ref{epsilon delta epsilon go to zero}),
$\left\vert \varepsilon s-\varepsilon\delta_{\varepsilon}\right\vert \leq\eta$ for
all $\varepsilon$ sufficiently small, and so by \eqref{estimate 5},
\[
c_{W}r^{n-1}+\eta\geq\left(  r-\eta\right)  ^{n-1} \int_{-b}^{b}\left(
W\left(  \check{w}_{\varepsilon}\left(  s\right)  \right)  +\left\vert \check
{w}_{\varepsilon}^{\prime}\left(  s\right)  \right\vert ^{2}\right)  \,ds\,.
\]
Since $\check{w}_{\varepsilon}\left(  0\right)  =0$, it follows that for all
$\varepsilon$ sufficiently small the sequence $\left\{  \check{w}_{\varepsilon
}\right\}  $ is bounded in $H^{1}\left(  -b,b\right)  $, and thus, up to a
subsequence not relabeled, it converges weakly in $H^{1}\left(  -b,b\right)  $ and uniformly to
some function $\check{w}\in H^{1}\left(  -b,b\right)  $. It follows by Fatou's
Lemma and the weak sequential lower semicontinuity of the $L^{2}$ norm that%
\begin{align}
c_{W}r^{n-1}+\eta &  \geq\left(  r-\eta\right)  ^{n-1} \limsup_{\varepsilon
\rightarrow0^{+}}\int_{-b}^{b}\left(  W\left(  \check{w}_{\varepsilon}\left(
s\right)  \right)  +\left\vert \check{w}_{\varepsilon}^{\prime}\left(  s\right)
\right\vert ^{2}\right)  \,ds\nonumber\\
&  \geq\left(  r-\eta\right)  ^{n-1} \liminf_{\varepsilon\rightarrow0^{+}}%
\int_{-b}^{b}\left(  W\left(  \check{w}_{\varepsilon}\left(  s\right)  \right)
+\left\vert \check{w}_{\varepsilon}^{\prime}\left(  s\right)  \right\vert
^{2}\right)  \,ds\label{estimate 6}\\
&  \geq\left(  r-\eta\right)  ^{n-1} \int_{-b}^{b}\left(  W\left(  \check{w}\left(  s\right)  \right)  +\left\vert \check{w}^{\prime}\left(  s\right)
\right\vert ^{2}\right)  \,ds\,.\nonumber
\end{align}
Letting $\eta\rightarrow0^{+}$ gives%
\begin{align*}
c_{W}  &  \geq\int_{-b}^{b}\left(  W\left(  \check{w}\left(  s\right)  \right)
+\left\vert \check{w}^{\prime}\left(  s\right)  \right\vert ^{2}\right)  \,ds\\
&  \geq\min_{w\in H^{1}\left(  -b,b\right)  \,, \,w\left(  0\right)  =0}\int%
_{-b}^{b}\left(  W\left(  w\left(  s\right)  \right)  +\left\vert w^{\prime
}\left(  s\right)  \right\vert ^{2}\right)  \,ds=c_{W}\,,
\end{align*}
where we have used \eqref{minimum cW}. 
Since $\check{w}\left(  0\right)  =0$, from the uniqueness of the minimizer it follows that $\check{w}=z$. Hence, the
entire sequence $\left\{  \check{w}_{\varepsilon}\right\}  $ weakly converges to
$z$ in $H^{1}\left(  -b,b\right)  $, and by (\ref{estimate 6}),
\[
c_{W}=\lim_{\varepsilon\rightarrow0^{+}}\int_{-b}^{b}\left(  W\left(  \check{w}_{\varepsilon}\left(  s\right)  \right)  +\left\vert \check{w}_{\varepsilon
}^{\prime}\left(  s\right)  \right\vert ^{2}\right)  \,ds=\int_{-b}^{b}\left(
W\left(  z\left(  s\right)  \right)  +\left\vert z^{\prime}\left(  s\right)
\right\vert ^{2}\right)  \,ds\,,
\]
which implies that
\[
\lim_{\varepsilon\rightarrow0^{+}}\int_{-b}^{b}\left\vert \check{w}%
_{\varepsilon}^{\prime}\left(  s\right)  \right\vert ^{2}\,ds=\int_{-b}%
^{b}\left\vert z^{\prime}\left(  s\right)  \right\vert ^{2}\,ds\,,
\]
and, in turn, \eqref{hat w epsilon to z} holds.

\medskip
\noindent {\bf Step 9.} {\it Here we prove that the sequence of Lagrange multipliers $\lambda_\varepsilon$ found in Step 4 converges to 
$\lambda_0:=(n-1)c_W$, where $c_W$ is defined in \eqref{constant cW}.} 

Note that a similar result was obtained in \cite{luckhaus-modica} in the case of a mass equality constraint. 

Since $\check{w}_{\varepsilon}\left(  -\tau_W\right)
\rightarrow-1$ and $\check{w}_{\varepsilon}\left(  {\tau_W}\right)  \rightarrow1^{-}$ by \eqref{hat w epsilon to z}, there exists a sequence $\left\{  \zeta_{\varepsilon}\right\}  $ of
positive numbers converging to $0$ such that $1+\check{w}_{\varepsilon}\left(
-{\tau_W}\right)  \leq\zeta_{\varepsilon}$ and $1-\check{w}_{\varepsilon}\left(
{\tau_W}\right)  \leq\zeta_{\varepsilon}$ for every $\varepsilon>0$. 
Then  for $\varepsilon>0$ small enough

\begin{align*}
\int_{-{\tau_W}}^{{\tau_W}}&\left(
W\left(  \check{w}_{\varepsilon}\left(  s\right)  \right)  +\left\vert \check
{w}_{\varepsilon}^{\prime}\left(  s\right)  \right\vert ^{2}\right)  \left(
r+\varepsilon s + \varepsilon\delta_{\varepsilon}\right)  ^{n-1} ds\\
&  \geq2\left(  r-\varepsilon {\tau_W} + \varepsilon\delta_{\varepsilon}\right)
^{n-1} \int_{-{\tau_W}}^{{\tau_W}}\left(  \sqrt{W\left(  \check{w}_{\varepsilon}\left(
s\right)  \right)  }\check{w}_{\varepsilon}^{\prime}\left(  s\right)  \right)
ds\\
&  =2\left(  r-\varepsilon {\tau_W} + \varepsilon\delta_{\varepsilon}\right)  ^{n-1}
\int_{\check{w}_{\varepsilon}\left(  -{\tau_W}\right)  }^{\check{w}_{\varepsilon}\left(
{\tau_W}\right)  }\sqrt{W\left(  \sigma\right)  }\,d\sigma\\
&  \geq2\left(  r-\varepsilon {\tau_W} + \varepsilon\delta_{\varepsilon}\right)
^{n-1} \int_{-1}^{1}\sqrt{W\left(  \sigma\right)  }\,d\sigma\\
&  -3r^{n-1} \int_{1- \zeta_{\varepsilon}}^{1}\sqrt{W\left(  \sigma\right)
}\,d\sigma-3r^{n-1} \int_{-1}^{-1+\zeta_{\varepsilon} }\sqrt{W\left(
\sigma\right)  }\,d\sigma\,.
\end{align*}
By (\ref{W beta})  we have
\[
\int_{1-\zeta_{\varepsilon}}^{1}\sqrt{W\left(  \sigma\right)  }\,d\sigma
=\int_{1-\zeta_{\varepsilon}}^{1}\left(  1-\sigma\right)  ^{\beta/2}%
\,d\sigma=\int_{0}^{\zeta_{\varepsilon}}s^{\beta/2}\,ds=\frac{2}{2+\beta}%
\zeta_{\varepsilon}^{(2+\beta)/2}%
\]
and, similarly,%
\[
\int_{-1}^{-1+\zeta_{\varepsilon} }\sqrt{W\left(  \sigma\right)  }%
\,d\sigma=\frac{2}{2+\beta}\zeta_{\varepsilon} ^{(2+\beta)/2}\,.
\]
In conclusion,%
\begin{align*}
     \int_{-\tau_W}^{\tau_W}&\left(
W\left(  \check{w}_{\varepsilon}\left(  s\right)  \right)  +\left\vert \check
{w}_{\varepsilon}^{\prime}\left(  s\right)  \right\vert ^{2}\right)  \left(
r+\varepsilon s + \varepsilon\delta_{\varepsilon}\right)  ^{n-1} ds\\
&  \geq c_{W}r^{n-1} -O\left(  \varepsilon \tau_W - \varepsilon\delta_{\varepsilon
}\right)  -\frac{12 r^{n-1}}{2+\beta} \zeta_{\varepsilon}^{(2+\beta)/2}\,.
\end{align*}
Fix $0<r_1<r<r_2<R$. Using (\ref{estimate 5}) and the previous inequality, we find that
\begin{align}
\int_{\tau_W}^{\frac{r_{2}}{\varepsilon}}  &  \left(  W\left(  \check{w}%
_{\varepsilon}\left(  s\right)  \right)  +\left\vert \check{w}_{\varepsilon
}^{\prime}\left(  s\right)  \right\vert ^{2}\right)  \left(  r+\varepsilon s +
\varepsilon\delta_{\varepsilon}\right)  ^{n-1} ds\nonumber\\
&  +\int_{-\frac{r_{1}}{\varepsilon}}^{-\tau_W}\left(  W\left(  \check{w}%
_{\varepsilon}\left(  s\right)  \right)  +\left\vert \check{w}_{\varepsilon
}^{\prime}\left(  s\right)  \right\vert ^{2}\right)  \left(  r+\varepsilon s +
\varepsilon\delta_{\varepsilon}\right)  ^{n-1} ds\label{eta epsilon}\\
&  \leq O\left(  \varepsilon^{2}\right)  + O\left(  \varepsilon \tau_W -
\varepsilon\delta_{\varepsilon}\right)  +\frac{12r^{n-1}}{2+\beta}
\zeta_{\varepsilon}^{(2+\beta)/2}=:\eta_{\varepsilon}\,.\nonumber
\end{align}
Note that, by \eqref{epsilon delta epsilon go to zero},
\begin{equation}\label{eta epsilon to zero}
 \lim_{\varepsilon\to 0+}\eta_\varepsilon=0\,.
\end{equation}

Fix $0<r_1^*<r_{1}$ and $0<r<r_2^*<r_{2}$. We claim that there exist
$a_{\varepsilon}\in\left(  -r_{1}/\varepsilon,-r_1^*/\varepsilon\right)  $ and $b_{\varepsilon}\in\left(  r_2^*/\varepsilon,r_{2}/\varepsilon\right)  $ such that
\begin{equation}
W\left(  \check{w}_{\varepsilon}\left(  a_{\varepsilon}\right)  \right)
+\left\vert \check{w}_{\varepsilon}^{\prime}\left(  a_{\varepsilon}\right)
\right\vert ^{2}\leq c_{1}\varepsilon\eta_{\varepsilon}\quad\text{and\quad
}W\left(  \check{w}_{\varepsilon}\left(  b_{\varepsilon}\right)  \right)
+\left\vert \check{w}_{\varepsilon}^{\prime}\left(  b_{\varepsilon}\right)
\right\vert ^{2}\leq c\varepsilon\eta_{\varepsilon} \label{estimate a and b}%
\end{equation}
for some appropriate constants $c_{1}=c_{1}\left(  r_{1}^*,r_1\right)  >0$ and
$c_{2}=c_{2}\left(  r_{2}^*,r_2\right)  >0$, independent of $\varepsilon$. To prove the existence of
$a_{\varepsilon}$, assume by contradiction that
\[
W\left(  \check{w}_{\varepsilon}\left(  s\right)  \right)  +\left\vert \check
{w}_{\varepsilon}^{\prime}\left(  s\right)  \right\vert ^{2}>c_{1}%
\varepsilon\eta_{\varepsilon}%
\]
for all $s\in\left(  -\frac{r_{1}}{\varepsilon},-\frac{r_1^*}{\varepsilon
}\right)  $. By (\ref{epsilon delta epsilon go to zero}), $\left\vert
\varepsilon\delta_{\varepsilon}\right\vert \leq\frac{1}{2}\left(
r-r_{1}\right)  $ for all $\varepsilon$ sufficiently small, and so, by \eqref{eta epsilon},
\begin{align*}
\eta_{\varepsilon}  &  \geq\int_{-\frac{r_{1}}{\varepsilon}}^{-\tau_W}\left(
W\left(  \check{w}_{\varepsilon}\left(  s\right)  \right)  +\left\vert \check
{w}_{\varepsilon}^{\prime}\left(  s\right)  \right\vert ^{2}\right)  \left(
r+\varepsilon s + \varepsilon\delta_{\varepsilon}\right)  ^{n-1} ds\\
&  \geq\int_{-\frac{r_{1}}{\varepsilon}}^{-\frac{r_1^*}{\varepsilon}}\left(
W\left(  \check{w}_{\varepsilon}\left(  s\right)  \right)  +\left\vert \check
{w}_{\varepsilon}^{\prime}\left(  s\right)  \right\vert ^{2}\right)  \left(
r+\varepsilon s + \varepsilon\delta_{\varepsilon}\right)  ^{n-1} ds\\
&  \geq c_{1}\eta_{\varepsilon}\left(  r-r_{1} + \varepsilon\delta
_{\varepsilon}\right)  ^{n-1} \left(  r_{1}-r_1^*\right) \\
&  \geq c_{1}\eta_{\varepsilon} \left(  \frac{r-r_{1}} {2 }\right)  ^{n-1}
\left(  r_{1}-r_1^*\right)  \,,
\end{align*}
which is a contradiction, provided we take
\[
c_{1}>\frac{2^{n-1}}{\left(  r-r_{1}\right)  ^{n-1} \left(  r_{1}%
-r_1^*\right)  }\,.
\]
This proves the existence of $a_{\varepsilon}$. The proof of the existence of
$b_{\varepsilon}$ is similar, and we omit it.

By (\ref{e-l}), $\check{w}_{\varepsilon}$ satisfies the Euler--Lagrange equation%
\[
-2\check{w}_{\varepsilon}^{\prime\prime}\left(  s\right)  -2\varepsilon\left(
n-1 \right)  \frac{\check{w}_{\varepsilon}^{\prime}\left(  s\right)
}{r+\varepsilon s + \varepsilon\delta_{\varepsilon}}+W^{\prime}\left(  \check
{w}_{\varepsilon}\left(  s\right)  \right)  =-\varepsilon\lambda_{\varepsilon
}\,.
\]
Multiplying the previous equation by $\check{w}_{\varepsilon}^{\prime}\left(
s\right)  $ gives%
\begin{equation}
\left(  -\left\vert \check{w}_{\varepsilon}^{\prime}\left(  s\right)
\right\vert ^{2}+W\left(  \check{w}_{\varepsilon}\left(  s\right)  \right)
\right)  ^{\prime}=2\varepsilon\left(  n-1 \right)  \frac{\left\vert \check
{w}_{\varepsilon}^{\prime}\left(  s\right)  \right\vert ^{2}}{r+\varepsilon s+
\varepsilon\delta_{\varepsilon}}-\varepsilon\lambda_{\varepsilon}\check
{w}_{\varepsilon}^{\prime}\left(  s\right)  \,. \label{e-l w hat}%
\end{equation}
Upon integration between $a_{\varepsilon}$ and $b_{\varepsilon}$ we get%
\begin{align*}
-  &  \left\vert \check{w}_{\varepsilon}^{\prime}\left(  b_{\varepsilon}\right)
\right\vert ^{2}+W\left(  \check{w}_{\varepsilon}\left(  b_{\varepsilon}\right)
\right)  +\left\vert \check{w}_{\varepsilon}^{\prime}\left(  a_{\varepsilon
}\right)  \right\vert ^{2}-W\left(  \check{w}_{\varepsilon}\left(
a_{\varepsilon}\right)  \right) \\
&  =2\varepsilon\left(  n-1 \right)  \int_{a_{\varepsilon}}^{b_{\varepsilon}%
}\frac{\left\vert \check{w}_{\varepsilon}^{\prime}\left(  s\right)  \right\vert
^{2}}{r+\varepsilon s + \varepsilon\delta_{\varepsilon}}\,ds-\varepsilon
\lambda_{\varepsilon}\left(  \check{w}_{\varepsilon}\left(  b_{\varepsilon
}\right)  -\check{w}_{\varepsilon}\left(  a_{\varepsilon}\right)  \right)  \,.
\end{align*}
Since $a_{\varepsilon}\in\left(  -r_{1}/\varepsilon,-r_1^*/\varepsilon\right)  $ and $b_{\varepsilon}\in\left(  r_2^*/\varepsilon,r_{2}/\varepsilon\right)  $ it follows from \eqref{convergence} and the monotonicity of $\check{w}_\varepsilon$  that $\check{w}_{\varepsilon}\left(
a_{\varepsilon}\right)  \rightarrow-1$ and $\check{w}_{\varepsilon}\left(
b_{\varepsilon}\right)  \rightarrow1$. Dividing the previous identity by
$\varepsilon$ and letting $\varepsilon\rightarrow0^{+}$,  by
(\ref{estimate a and b}) we get%
\begin{equation}
\lim_{\varepsilon\rightarrow0^{+}}\left(  2 \left(  n-1 \right)
\int_{a_{\varepsilon}}^{b_{\varepsilon}}\frac{\left\vert \check{w}_{\varepsilon
}^{\prime}\left(  s\right)  \right\vert ^{2}}{r+\varepsilon s+ \varepsilon
\delta_{\varepsilon}}\,ds-\lambda_{\varepsilon}\right)  \,=0\,.
\label{lambda go to 1}%
\end{equation}
Observe that by (\ref{eta epsilon}),%
\begin{align*}
\int_{\tau_{W}}^{b_{\varepsilon}}\frac{\left\vert \check{w}_{\varepsilon
}^{\prime}\left(  s\right)  \right\vert ^{2}}{r+\varepsilon s + \varepsilon
\delta_{\varepsilon}}\,ds  &  \leq\frac{1}{\left(  r + \varepsilon
\delta_{\varepsilon}\right)  ^{n}}\int_{\tau_{W}}^{b_{\varepsilon}}\left\vert
\check{w}_{\varepsilon}^{\prime}\left(  s\right)  \right\vert ^{2}\left(
r+\varepsilon s + \varepsilon\delta_{\varepsilon}\right)^{n-1}  \,ds\\
&  \leq\frac{\eta_{\varepsilon}}{\left(  r + \varepsilon\delta_{\varepsilon
}\right)  ^{n}}\rightarrow0
\end{align*}
where the convergence to zero follows from  \eqref{epsilon delta epsilon go to zero} and \eqref{eta epsilon to zero}. Similarly,
\[
\int_{a_{\varepsilon}}^{-\tau_{W}}\frac{\left\vert \check{w}_{\varepsilon
}^{\prime}\left(  s\right)  \right\vert ^{2}}{r+\varepsilon s + \varepsilon
\delta_{\varepsilon}}\,ds\leq\frac{\eta_{\varepsilon}}{\left(  r-r_{1} +
\varepsilon\delta_{\varepsilon}\right)  ^{n}}\rightarrow0\,.
\]
These inequalities, together with (\ref{hat w epsilon to z}) and
(\ref{lambda go to 1}), imply that
\begin{equation}
\lim_{\varepsilon\rightarrow0^{+}}\lambda_{\varepsilon}=\lambda_{0}:=2 \left(
n-1 \right)  \int_{- \tau_{W}}^{ \tau_{W}}\left\vert z^{\prime}\left(
s\right)  \right\vert ^{2}\,ds = \left(  n-1 \right)  c_{W}\,,
\label{limit lambda}%
\end{equation}
where in the last equality we have used (\ref{constant cW 2}).

\medskip
\noindent {\bf Step 10.} {\it We claim that there exist two constants $a_{1}>0$ and $c_{1}>0$ such that%
\begin{equation}
\check{w}_{\varepsilon}\left(  -\tau_{W}-a_{1}\varepsilon\right)  \leq
-1+c_{1}\varepsilon^{1/\beta} \label{rough estimate}%
\end{equation}
for all $\varepsilon>0$ sufficiently small.
}

Let $b_\varepsilon$ be the number given in \eqref{estimate a and b} and let $t<b_\varepsilon$. 
Integrating (\ref{e-l w hat}) between $t$ and $b_{\varepsilon}$ and using
\eqref{bound w epsilon} and (\ref{limit lambda})  gives%
\begin{align*}
&  -\left\vert \check{w}_{\varepsilon}^{\prime}\left(  b_{\varepsilon}\right)
\right\vert ^{2}+W\left(  \check{w}_{\varepsilon}\left(  b_{\varepsilon}\right)
\right)  +\left\vert \check{w}_{\varepsilon}^{\prime}\left(  t\right)
\right\vert ^{2}-W\left(  \check{w}_{\varepsilon}\left(  t\right)  \right) \\
&  =2\varepsilon\left(  n-1 \right)  \int_{t}^{b_{\varepsilon}}\frac
{\left\vert \check{w}_{\varepsilon}^{\prime}\left(  s\right)  \right\vert ^{2}%
}{r+\varepsilon s-\varepsilon\delta_{\varepsilon}}\,ds -\varepsilon
\lambda_{\varepsilon}\left(  \check{w}_{\varepsilon}\left(  b_{\varepsilon
}\right)  -\check{w}_{\varepsilon}\left(  t\right)  \right)  \geq-c\varepsilon
\end{align*}
for some constant $c>0$ independent of $\varepsilon$. In view of (\ref{estimate a and b}), by taking $c>0$
larger, if necessary, we have%
\[
\left\vert \check{w}_{\varepsilon}^{\prime}\left(  t\right)  \right\vert
^{2}>W\left(  \check{w}_{\varepsilon}\left(  t\right)  \right)  -c\varepsilon
\,.
\]
Since $\check{w}_{\varepsilon}$ is nondecreasing, we obtain that
\[
\check{w}_{\varepsilon}^{\prime}\left(  t\right)  \geq\sqrt{(  W(
\check{w}_{\varepsilon}(  t)  )  -c\varepsilon)  ^{+}%
}\,.
\]
 
 Using the fact that $W$ is continuous and vanishes only at $s=\pm 1$, we find a constant $\mu>0$ such that
\begin{equation}\label{W larger mu}
 W(s)\ge \mu\quad \text{for }-1+a\le s\le 1-a\,,
\end{equation}
where $a$ is the constant in \eqref{W beta}. Therefore, if $c\varepsilon<\mu$ and $\left(  c\varepsilon\right)  ^{1/\beta}< a$ we deduce from \eqref{W beta} that 
$W(s)>c\varepsilon$ for $-1+\left(  c\varepsilon\right)  ^{1/\beta}<s\le 0$. This implies
\[
\check{w}_{\varepsilon}^{\prime}(t)\ge\sqrt{ W\left(  \check{w}_{\varepsilon}\left(  t\right)
\right)  -c\varepsilon}\,.
\]
 for every $-\frac{r}{\varepsilon} -\delta_\varepsilon<t\le 0$ such
that $-1+\left(  c\varepsilon\right)  ^{1/\beta} <\check{w}_{\varepsilon}\left(
t\right)  \le 0$.

Since $\check{w}_{\varepsilon}\left(  0\right)  =0$, it follows upon integration
\[
-t\le \int_{\check{w}_{\varepsilon}(t)}^{0} \frac{1}{\sqrt{W\left(  s \right)
-c\varepsilon}}\,ds\le \int_{-1+\left(  c\varepsilon\right)^{1/\beta}}^{0} \frac{1}{\sqrt{W\left(  s \right)
-c\varepsilon}}\,ds<+\infty\,,
\]
where the last inequality is a consequence of \eqref{W beta}. 
Observe that this inequality provides a bound on all those $t$ such that 
$-1+\left(  c\varepsilon\right)^{1/\beta} <\check{w}_{\varepsilon}\left(
t\right)  \le 0$. Therefore, there exists a largest $t_\varepsilon$ such that 
$-\frac{r}{\varepsilon} -\delta_\varepsilon<t_{\varepsilon}<0$ and $\check{w}_{\varepsilon}\left(
t_{\varepsilon}\right)  =-1+\left(  c\varepsilon\right)  ^{1/\beta}$. Then
\begin{align}
-t_{\varepsilon}  &  \le \int_{-1+ \left(  c\varepsilon\right)  ^{1/\beta}}^{0}
\frac{1}{\sqrt{W\left(  s \right)  -c\varepsilon}}\,ds\nonumber\\
&  =\int_{-1+ \left(  c\varepsilon\right)  ^{1/\beta}}^{-1+a} \frac{1}%
{\sqrt{(  s+1 )  ^{\beta} -c\varepsilon}}\,ds + \int_{-1+ a}^{0}
\frac{1}{\sqrt{W\left(  s \right)  -c\varepsilon}}\,ds \,.
\label{estimate t epsilon}%
\end{align}

The change of variables $\sigma:=\left(  s+1 \right)  ^{\beta}
-c\varepsilon$ yields
\begin{align}
&  \int_{-1+ \left(  c\varepsilon\right)  ^{1/\beta}}^{-1+a} \frac{1}%
{\sqrt{(  s+1 )  ^{\beta} -c\varepsilon}}\,ds= \frac{1}{\beta}%
\int_{0}^{a^{\beta}-c\varepsilon} \frac{1}{\sigma^{1/2}\left(  \sigma
+c\varepsilon\right)  ^{1-1/\beta}}\,d\sigma\nonumber\\
&  <\frac{1}{\beta}\int_{0}^{a^{\beta}}\frac{1}{\sigma^{1/2}\sigma^{1-1/\beta
}}\,d\sigma= \int_{-1}^{-1+ a} \frac{1}{\sqrt{W(  s )  }}\,ds\,,
\label{estimate 44}%
\end{align}
where the last equality follows from direct computation, taking into account
(\ref{W beta}).

By \eqref{W larger mu}  there exists
a constant $L>0$ such that
\begin{equation}
\left\vert \frac{1}{\sqrt{W\left(  s \right)  -\eta}} - \frac{1}%
{\sqrt{W\left(  s \right)  }} \right\vert \leq L\eta
\label{Lipschitz estimate}%
\end{equation}
for $-1+a\leq s\leq 1-a$ and for $0<\eta\leq \mu/2$.

From (\ref{estimate t epsilon}), (\ref{estimate 44}), and
(\ref{Lipschitz estimate}) we get
\[
-t_{\varepsilon}\leq\int_{-1}^{0} \frac{ds}{\sqrt{W\left(  s \right)  }} +
Lc\varepsilon= \tau_{W}+ Lc\varepsilon\,,
\]
where the equality follows from (\ref{tau W}). Since $\check{w}_{\varepsilon}$ is
nondecreasing, we obtain $\check{w}_{\varepsilon}\left(  -\tau_{W}- Lc%
\varepsilon\right)  \leq \check{w}_{\varepsilon}\left(  t_{\varepsilon} \right)
=-1+\left(  c\varepsilon\right)  ^{1/\beta}$, which gives (\ref{rough estimate}) with  
$a_1:=Lc$ and 
$c_{1}:=c^{1/\beta}$.

\medskip
\noindent {\bf Step 11.} {\it  Let $1<\alpha\le 2$ be defined by
\begin{equation}
\alpha:=\frac{1}{\frac{1}{2}+|\beta-\frac{3}{2}|}\,. 
\nonumber
\end{equation}
We claim
there exist two constants $a_2>0$ and $c_{2}>0$
such that%
\begin{equation}\label{one sided lemma}
\check{w}_{\varepsilon}\left(  \tau_{W} + a_{2}\varepsilon\right)  \geq
1-c_{2}\varepsilon^{\alpha}
\end{equation}
for all $\varepsilon>0$ sufficiently small.
}

Since $\check{w}_{\varepsilon}\left(
\frac{R-r}{\varepsilon} - \delta_{\varepsilon}\right)  =1$, integrating (\ref{e-l w hat}) between $t$ and $\frac{R-r}{\varepsilon}-
\delta_{\varepsilon}$  we obtain%
\begin{align*}
&  \left\vert \check{w}_{\varepsilon}^{\prime}\left(  t\right)  \right\vert
^{2}-W\left(  \check{w}_{\varepsilon}\left(  t\right)  \right)  +\varepsilon
\lambda_{\varepsilon}\left(  1-\check{w}_{\varepsilon}\left(  t\right)  \right)
\\
&  =\left\vert \check{w}_{\varepsilon}^{\prime}\left(  \tfrac{R-r}{\varepsilon} -
\delta_{\varepsilon}\right)  \right\vert ^{2}+2\varepsilon\left(  n-1 \right)
\int_{t}^{\frac{R-r}{\varepsilon} - \delta_{\varepsilon}}\frac{\left\vert
\check{w}_{\varepsilon}^{\prime}\left(  s\right)  \right\vert ^{2}%
}{r+\varepsilon s + \varepsilon\delta_{\varepsilon}}\,ds\geq0\,.
\end{align*}
Hence, by (\ref{limit lambda})%
\[
\left\vert \check{w}_{\varepsilon}^{\prime}\left(  t\right)  \right\vert ^{2}
\geq W\left(  \check{w}_{\varepsilon}\left(  t\right)  \right)  -c\varepsilon
\left(  1-\check{w}_{\varepsilon}\left(  t\right)  \right) 
\]
for all $\varepsilon>0$ sufficiently small  and for some constant $c>0$ 
independent of $\varepsilon$.
Since $\check{w}_{\varepsilon}$ is nondecreasing, we deduce that
\[
\check{w}_{\varepsilon}^{\prime}\left(  t\right)  \geq\sqrt{(  W
\check{w}_{\varepsilon}(  t)  )  -c\varepsilon(  1-\check
{w}_{\varepsilon}(  t)  )  )  ^{+}}\,.
\]
For $t>0$ sufficiently small, we have that $0\le\check{w}_{\varepsilon}\left(
t\right)  <1$. 
Since $\check{w}_{\varepsilon}\left(  0\right)  =0$, it  follows upon integration
\[
t\le \int_{0}^{\check{w}_{\varepsilon}(t)}\frac{ds}{\sqrt{(  W(  s )
-c\varepsilon (  1-s)  )  ^{+}}}\,.
\]

Let $\gamma:=\min\{2(\beta-1),1\}$. Note that $\alpha=\frac{1}{\beta-\gamma}$.
Then $\left(  W\left(  s \right)  -c\varepsilon\left(  1-s \right)  \right)
^{+}\geq\left(  W\left(  s \right)  -c\varepsilon\left(  1-s \right)
^{\gamma}\right)  ^{+}$ for $0\leq s\leq1$, hence
\[
t\leq\int_{0}^{\check{w}_{\varepsilon}(t)}\frac{ds}{\sqrt{(  W(  s )
-c\varepsilon(  1-s )  ^{\gamma})  ^{+}}}\,.
\]
By \eqref{W larger mu}  when
$c\varepsilon\le \mu$ we have $W\left(  s \right)  -c\varepsilon\left(  1-s
\right)  ^{\gamma}\geq0$ for $  0\le s\le 1-a  $.  Moreover, if  $\left(  c\varepsilon\right)
^{\alpha}<a$ we obtain 
$W\left(  s \right)  =\left(  1-s \right)  ^{\beta}$ for $1-a\leq
s\leq1-\left(  c\varepsilon\right)  ^{\alpha}$ by \eqref{W beta}.

Using the facts that $\check{w}_{\varepsilon}(0)=0$ and $\check{w}_{\varepsilon}\left(
\frac{R-r}{\varepsilon} - \delta_{\varepsilon}\right)  =1$, there exists
$0<t_{\varepsilon}<\frac{R-r}{\varepsilon} - \delta_{\varepsilon}$ such that $\check{w}_{\varepsilon}\left(  t_{\varepsilon
}\right)  = 1-\left(  c\varepsilon\right)  ^{\alpha}$. Then for $\varepsilon$ sufficiently small
\begin{equation}
t_{\varepsilon}\leq\int_{0}^{1-a} \!\!\!\!\!\! \frac{ds}{\sqrt{ W\left(  s
\right)  -c\varepsilon\left(  1-s \right)  ^{\gamma}}}+ \int_{1-a}^{1-\left(
c\varepsilon\right)  ^{\alpha}} \!\!\!\!\!\!\!\!\!\! \frac{ds}{\left(  1-s
\right)  ^{\gamma/2} \sqrt{(  1-s )  ^{\beta-\gamma}-c\varepsilon}%
}\,. \label{new estimate t epsilon}%
\end{equation}
Consider the change of variables $\sigma:=\left(  1-s \right)  ^{\beta-\gamma
}-c\varepsilon$. Then $\left(  \sigma+c\varepsilon\right)  ^{1/\left(
\beta-\gamma\right)  }=1-s$, and since $2-2\beta+\gamma\leq0$ and $\beta-\gamma\ge 0$,
\begin{align}
&  \int_{1-a}^{1-\left(  c\varepsilon\right)  ^{\alpha}} \!\!\!\!\!\!\!\!\!\!
\frac{ds}{\left(  1-s \right)  ^{\gamma/2} \sqrt{(  1-s )
^{\beta-\gamma}-c\varepsilon}}\nonumber\\
&  \leq\frac{1}{\beta-\gamma}\int_{0}^{a^{\beta-\gamma}-c\varepsilon} \left(
\sigma+c\varepsilon\right)  ^{\left(  2-2\beta+\gamma\right)  /\left(
2\beta-2\gamma\right)  } \sigma^{-1/2}\,d\sigma\nonumber\\
&  \leq\frac{1}{\beta-\gamma}\int_{0}^{a^{\beta-\gamma}} \sigma^{\left(
2-2\beta+\gamma\right)  /\left(  2\beta-2\gamma\right)  } \sigma
^{-1/2}\,d\sigma\nonumber\\
&  =\frac{2}{2-\beta} a^{\frac{2-\beta}{2}}= \int_{1-a}^{1} \frac{ds}%
{\sqrt{W(  s )  }} \,, \label{estimate 47}%
\end{align}
where the last inequality follows from direct computation, taking into account
(\ref{W beta}).

From (\ref{Lipschitz estimate}), (\ref{new estimate t epsilon}), and
(\ref{estimate 47}) we obtain that for $\varepsilon>0$ small enough we have
\[
t_{\varepsilon}\leq\int_{0}^{1} \frac{ds}{\sqrt{W\left(  s \right)  }}+
Lc\varepsilon= \tau_{W}+ Lc\varepsilon\,,
\]
where the equality is a consequence of  (\ref{tau W}). Since $\check{w}_{\varepsilon}$ is
nondecreasing, it follows that $\check{w}_{\varepsilon
}\left(  \tau_{W}+ Lc\varepsilon\right)   \geq \check{w}_{\varepsilon}\left(  t_{\varepsilon
}\right)  = 1-\left(  c\varepsilon\right)  ^{\alpha}$. This concludes the proof of \eqref{one sided lemma} with $a_2:=Lc$ and  $c_2:=c^\alpha$.

\medskip
\noindent {\bf Step 12.} {\it Let $\delta_\varepsilon$ be the constant introduced in Step 7 and let $w_0$ be the function introduced in \eqref{w0}. We claim that
\begin{equation}
\liminf_{\varepsilon\to0+}\int_{\delta_{\varepsilon}+\tau_{W}}^{\frac
{R-r}{\varepsilon}}\left(  w_{\varepsilon}\left(  t\right)  -w_{0}\left(
t\right)  \right)  \left(  r+\varepsilon t\right)  ^{n-1} dt\geq0\,.
\label{volume tail}%
\end{equation}
}

Since $w_{\varepsilon}$ is nondecreasing and $w_{\varepsilon}(\delta
_{\varepsilon}+\tau_{W}+a_2\varepsilon)\geq1-c_{2}\varepsilon^{\alpha}$ by \eqref{one sided lemma}, we have $w_{\varepsilon}(t)\geq1-c_{2}\varepsilon
^{\alpha}>0$ for every $t\geq\delta_{\varepsilon}+\tau_{W}+a_2\varepsilon$. This implies
that
\begin{align}
&  \int_{\delta_{\varepsilon}+\tau_{W}}^{\frac{R-r}{\varepsilon}}\!\!\!\!\!\!\!\left(
w_{\varepsilon}\left(  t\right)  -w_{0}\left(  t\right)  \right)  \left(
r+\varepsilon t\right)  ^{n-1} dt=\int_{\delta_{\varepsilon}+\tau_{W}}^{\delta_{\varepsilon}+\tau_{W}+a_2\varepsilon}\!\!\!\!\!\!\!\!\!\!\!\!\!\!\!\!\!\!\!\!\left(
w_{\varepsilon}\left(  t\right)  -w_{0}\left(  t\right)  \right)  \left(
r+\varepsilon t\right)  ^{n-1} dt \nonumber\\
&+\int_{\delta_{\varepsilon}+\tau_{W}+a_2\varepsilon}^{\left(  \delta_{\varepsilon}%
+\tau_{W}+a_2\varepsilon\right)  ^{+}}\!\!\!\!\!\!\!\!\!\!\!\!\!\!\!\!\!\!\!\!\!\!\!\!\!\!\!\left(  w_{\varepsilon}\left(  t\right)  +1 \right)
\left(  r+\varepsilon t\right)  ^{n-1} dt + \int_{\left(  \delta_{\varepsilon
}+\tau_{W}+a_2\varepsilon\right)  ^{+}}^{\frac{R-r}{\varepsilon}}\!\!\!\!\!\!\!\!\!\!\!\!\!\!\!\!\!\!\!\!\!\!\!\!\!\left(  w_{\varepsilon
}\left(  t\right)  -1 \right)  \left(  r+\varepsilon t\right)  ^{n-1} dt\nonumber\\
& \geq-c_{2} \varepsilon^{\alpha} \int_{\left(  \delta_{\varepsilon}+\tau
_{W}\right)  ^{+}}^{\frac{R-r}{\varepsilon}} \!\!\!\!\!\!\!\!\!\!\!\!\!\!\!\left(  r+\varepsilon t\right)
^{n-1} dt +O(\varepsilon)\geq-c_{2} \varepsilon^{\alpha-1} \frac{R^{n}-r^{n}}{n} +O(\varepsilon)\,,
\label{9000}
\end{align}
where we used \eqref{epsilon delta epsilon go to zero} together with the inequality
\[
 \int_{\delta_{\varepsilon}+\tau_{W}}^{\delta_{\varepsilon}+\tau_{W}+a_2\varepsilon}
 \!\!\!\!\!\!\!\!\!\!\!\!\!\!\!\!\!\!\!\!\!\!\!\left(
w_{\varepsilon}\left(  t\right)  -w_{0}\left(  t\right)  \right)  \left(
r+\varepsilon t\right)  ^{n-1} dt\ge -3\left(r+\varepsilon(|\delta_\varepsilon|+\tau_W+a_2\varepsilon)\right)^{n-1}a_2\varepsilon\,,
 \]
 which follows from \eqref{bound w epsilon}.

Since $\alpha>1$, inequality \eqref{9000} yields (\ref{volume tail}).

\medskip
\noindent {\bf Step 13.} {\it 
 We claim that
\begin{equation}
\liminf_{\varepsilon\to0+}\delta_{\varepsilon} \geq0\,. \label{delta geq 0}%
\end{equation}
}

Assume, by contradiction, that there exists a sequence $\varepsilon
_{j}\rightarrow0+$ such that $\delta_{\varepsilon_{j}} \to\delta_{0}$ for
some $\delta_{0}$ satisfying $-\infty\leq\delta_{0}<0$. This implies, in
particular, that $\delta_{\varepsilon_{j}}-\tau_{W}<0$ for $j$ large enough.
By (\ref{inf w pound}) it follows that
\begin{align*}
&  \int_{-\frac{r}{\varepsilon_{j}}}^{\delta_{\varepsilon_{j}}-\tau_{W}%
}\!\!\!\!\!\!\!\!\!\!\!\!\!\!\!\! (  w_{\varepsilon_{j} }\left(  t\right)  -w_{0}\left(  t\right)
)  \left(  r+\varepsilon_{j} t\right)  ^{n-1} dt =\int_{-\frac
{r}{\varepsilon_{j}}}^{\delta_{\varepsilon_{j}}-\tau_{W}}\!\!\!\!\!\!\!\!\!\!\!\!\!\!\!\!(
w_{\varepsilon_{j} }\left(  t\right)  +1 )  \left(  r+\varepsilon_{j}
t\right)  ^{n-1} dt\\
&  \geq-\left(  \frac{\varepsilon_{j}\lambda_{\varepsilon_{j}}}{\beta}\right)
^{1/\left(  \beta-1\right)  }\frac{r^{n}}{n\,\varepsilon_{j}}=-\frac
{\varepsilon_{j}^{(2-\beta)/(\beta-1)} \lambda_{\varepsilon_{j}}^{1/\left(
\beta-1\right)  }r^{n}}{\beta^{1/\left(  \beta-1\right)  } \,n}\,.
\end{align*}
Since $1<\beta<2$, by (\ref{limit lambda}) we obtain
\[
\liminf_{j\to\infty}\int_{-\frac{r}{\varepsilon_{j}}}^{\delta_{\varepsilon
_{j}}-\tau_{W}}\left(  w_{\varepsilon_{j} }\left(  t\right)  -w_{0}\left(
t\right)  \right)  \left(  r+\varepsilon_{j} t\right)  ^{n-1} dt \geq0\,.
\]

This inequality, together with (\ref{rescaled volume constraint}) and
(\ref{volume tail}), implies that
\[
\limsup_{j\to\infty}\int_{\delta_{\varepsilon_{j}}-\tau_{W}}^{\delta
_{\varepsilon_{j}}+\tau_{W}}\left(  w_{\varepsilon_{j} }\left(  t\right)
-w_{0}\left(  t\right)  \right)  \left(  r+\varepsilon_{j} t\right)  ^{n-1} dt
\leq0\,.
\]
Changing variables, we get
\[
\limsup_{j\to\infty}\int_{-\tau_{W}}^{\tau_{W}}\left(  \check w_{\varepsilon_{j}
}\left(  t\right)  -w_{0}\left(  t+\delta_{\varepsilon_{j}}\right)  \right)
\left(  r+\varepsilon_{j} t + \varepsilon_{j}\delta_{\varepsilon_{j}} \right)
^{n-1} dt \leq0\,.
\]
Since $\varepsilon_{j}\delta_{\varepsilon_{j}} \to0$ by
(\ref{epsilon delta epsilon go to zero}), we conclude that
\begin{equation}\label{700}
 \limsup_{j\to\infty}\int_{-\tau_{W}}^{\tau_{W}}\left(  \check w_{\varepsilon_{j}
}\left(  t\right)  -w_{0}\left(  t+\delta_{\varepsilon_{j}}\right)  \right)
\,dt \leq0\,.
\end{equation}
Using the fact that $\check w_{\varepsilon_{j}}\to z$ strongly in $H^{1}(-\tau_{W},\tau_{W})$, from the
previous inequality we obtain
\begin{equation}
\int_{-\tau_{W}}^{\tau_{W}}\big( z\left(  t\right)  -w_{0}\big( t+\hat
\delta_{0}\big) \big) \,dt \leq0\,, \label{limit volume constraint}%
\end{equation}
where $\hat\delta_{0}:=\max\{\delta_{0},-\tau_{W}\}$. 
Indeed, if $\delta_0\ge -\tau_W$ we pass to the limit in \eqref{700}. If 
$-\infty<\delta_0<-\tau_W$ we use also the equalities 
$w_0(t+\delta_0)=w_0(t+\hat\delta_0)=-1$ for $t\in[-\tau_W,\tau_W]$  
in the case $\delta_0< -\tau_W$. The case $\delta_0=-\infty$ is similar.
Since
\[
\int_{-\tau_{W}}^{\tau_{W}} z\left(  t\right)  \,dt =0 \qquad\hbox{and}\qquad
\int_{-\tau_{W}}^{\tau_{W}}w_{0}( t+\hat\delta_{0}) \,dt =2\hat\delta_{0} \,,
\]
from (\ref{limit volume constraint}) we obtain $2\hat\delta_{0}\geq0$, which
contradicts the inequality $\delta_{0}<0$. This completes the proof of \eqref{delta geq 0}.

\bigskip\medskip
\noindent {\bf Step 14.} {\it We conclude the proof of the theorem by showing that
\eqref{Gamma liminf ineq 2} holds.  }

Let $a_1$ and $a_2$ be the constants given in \eqref{Lipschitz estimate} and \eqref{one sided lemma}. 
By (\ref{Gepsilon}),\eqref{epsilon delta epsilon go to zero}, and \eqref{w check} we have
\begin{align}
\mathcal{H}_{\varepsilon}(w_{\varepsilon})  &  \geq\int_{\delta_{\varepsilon
}-\tau_{W}-a_1\varepsilon}^{\delta_{\varepsilon}+\tau_{W}+a_2\varepsilon} \left(  W\left(  w_{\varepsilon
}\left(  t \right)  \right)  + \left\vert w_{\varepsilon}^{\prime}\left(
t\right)  \right\vert ^{2}\right)  \left(  r+\varepsilon t \right)  ^{n-1}
dt\nonumber\\
&  =\int_{-\tau_{W}-a_1\varepsilon}^{\tau_{W}+a_2\varepsilon}\left(  W\left(  \check{w}_{\varepsilon}\left(  s
\right)  \right)  + \left\vert \check{w}_{\varepsilon}^{\prime}\left(  s\right)
\right\vert ^{2}\right)  \left(  r+\varepsilon s + \varepsilon\delta_{\varepsilon
}\right)  ^{n-1} ds\label{estimate 10}\\
\vphantom{\int_{-\tau_{W}}^{\tau_{W}}}  &  \geq I^{1}_{\varepsilon}
+\varepsilon I^{2}_{\varepsilon} + \varepsilon\delta_{\varepsilon}
I^{3}_{\varepsilon} +\varepsilon I^{4}_{\varepsilon}\,,\nonumber
\end{align}
where
\begin{align*}
I^{1}_{\varepsilon}:=  &  \ 2 r^{n-1} \int_{-\tau_{W}-a_1\varepsilon}^{\tau_{W}+a_2\varepsilon}%
\sqrt{W\left(  \check{w}_{\varepsilon}\left(  s \right)  \right)  } \left\vert
\check{w}_{\varepsilon}^{\prime}\left(  s\right)  \right\vert \,ds\,,\\
I^{2}_{\varepsilon}:=  &  \ \left(  n-1 \right)  r^{n-2} \int_{-\tau_{W}-a_1\varepsilon%
}^{\tau_{W}+a_2\varepsilon}\left(  W\left(  \check{w}_{\varepsilon}\left(  s \right)  \right)
+ \left\vert \check{w}_{\varepsilon}^{\prime}\left(  s\right)  \right\vert
^{2}\right)  s\,ds\,,\\
I^{3}_{\varepsilon}:=  &  \ \left(  n-1 \right)  r^{n-2} \int_{-\tau_{W}-a_1\varepsilon%
}^{\tau_{W}+a_2\varepsilon}\left(  W\left(  \check{w}_{\varepsilon}\left(  s \right)  \right)
+ \left\vert \check{w}_{\varepsilon}^{\prime}\left(  s\right)  \right\vert
^{2}\right)  \,ds\,,\\
I^{4}_{\varepsilon}:=  &  \ \sum_{h=2}^{n-1} \binom{n-1}{h} r^{n-1-h}%
\varepsilon^{h-1}J_{h,\varepsilon}\,,
\end{align*}
with
\[
J_{h,\varepsilon}:= \int_{-\tau_{W}-a_1\varepsilon}^{\tau_{W}+a_2\varepsilon}\left(  W\left(  \check
{w}_{\varepsilon}\left(  s \right)  \right)  + \left\vert \check{w}%
_{\varepsilon}^{\prime}\left(  s\right)  \right\vert ^{2}\right)  \left(
s + \delta_{\varepsilon}\right)  ^{h} ds\,.
\]
To estimate $I^{1}_{\varepsilon}$ we use \eqref{cW} to write
\begin{align}
I^{1}_{\varepsilon}&=    \ 2 r^{n-1} \int_{\check{w}_{\varepsilon} \left(
-\tau_{W}-a_1\varepsilon\right)  }^{\check{w}_{\varepsilon}\left(  \tau_{W}+a_2\varepsilon\right)  }%
\sqrt{W\left(  \sigma\right)  } \,d\sigma\nonumber\\
& =c_{W} r^{n-1} -2 r^{n-1} \int_{-1}^{\check{w}_{\varepsilon} \left(  -\tau_{W}-a_1\varepsilon\right)  }\!\!\!\!\!\!\!\!\!\!\!\!\!\!\!\!\!\!\!\!\!\!\!\!\!\!\!\!\!\!
\sqrt{W\left(  \sigma\right)  }\,d\sigma-2 r^{n-1} \int_{\check{w}_{\varepsilon}
\left(  \tau_{W}+a_2\varepsilon\right)  }^{1} \!\!\!\!\!\!\!\!\!\!\!\!\!\!\!\!\!\!\!\!\!\!\!\!\sqrt{W\left(  \sigma\right)  }\,d\sigma\,.
\label{estimate 104}%
\end{align}

By \eqref{W beta} and (\ref{rough estimate}) we have
\begin{align*}
&  2 r^{n-1}\int_{-1}^{\check{w}_{\varepsilon} \left(  -\tau_{W}-a_1\varepsilon\right)  }
\!\!\!\!\!\!\!\!\!\!\sqrt{W\left(  \sigma\right)  }\,d\sigma\leq2 r^{n-1}\int_{-1}^{-1+c_{1}%
\varepsilon^{1/\beta}}\!\!\!\!\! \sqrt{W\left(  \sigma\right)  }\,d\sigma\\
&  = 2 r^{n-1}\int_{0}^{c_{1}\varepsilon^{1/\beta}}\!\!\!\!\! s^{\beta/2}\,ds
=k_{1} \varepsilon^{(\beta+2)/(2\beta)} \,,
\end{align*}
where $k_{1}:=\frac{4}{\beta+2} r^{n-1} c_{1}^{(\beta+2)/2}$. Similarly, 
by \eqref{one sided lemma} we can find a constant $k_2>0$ such that 
\[
\int_{\check{w}_{\varepsilon} \left(  \tau_{W}+a_2\varepsilon\right)  }^{1}\!\!\!\!\!\sqrt{W\left(
\sigma\right)  }\,d\sigma\leq k_{2} \varepsilon^{(\beta+2)/(2\beta
)}\,,
\]
where we have used the fact that $\alpha>1/\beta$.

Therefore (\ref{estimate 104}) gives
\begin{equation}
I^{1}_{\varepsilon} \geq c_{W}r^{n-1} -k_{0} \varepsilon^{(\beta+2)/(2\beta
)}\, \label{estimate 11}%
\end{equation}
where $k_0:=k_1+k_2$.

By \eqref{constant cW}, \eqref{bound w epsilon}, and (\ref{hat w epsilon to z}) we have
\begin{align}
\lim_{\varepsilon\rightarrow0+} I^{2}_{\varepsilon}=  &  \left(  n-1 \right)
r^{n-2} \!\int_{-\tau_{W}}^{\tau_{W}}\!\!\!\left(  W\left(  z\left(  s
\right)  \right)  + \left\vert z^{\prime}\left(  s\right)  \right\vert
^{2}\right)  \! s\,ds=0\,,\label{estimate 106}\\
\lim_{\varepsilon\rightarrow0+} I^{3}_{\varepsilon}=  &  \left(  n-1 \right)
r^{n-2} \! \int_{-\tau_{W}}^{\tau_{W}}\!\!\!\left(  W\left(  z\left(  s
\right)  \right)  + \left\vert z^{\prime}\left(  s\right)  \right\vert
^{2}\right)  \!ds=\left(  n-1 \right)  c_{W} r^{n-2}, \label{estimate 107}%
\end{align}
where to obtain the second equality in (\ref{estimate 106}) we have used the fact that $z$ is odd and $W$ is even .

Moreover, $J_{h,\varepsilon}\ge 0$ if $h$ is even, while if $h$ is odd
by Fatou's Lemma,   
(\ref{hat w epsilon to z}), and (\ref{delta geq 0})  we obtain
\[
 \liminf_{\varepsilon\to 0+} J_{h,\varepsilon}\ge 
 \int_{-\tau_{W}}^{\tau_{W}}\left(  W\left(  
z\left(  s \right)  \right)  + \left\vert z^{\prime}(s)\right\vert^2\right) s^h ds=0\,,
\]
and the integral is zero because $z$ is odd and $W$ is even. Hence,
\begin{equation}
\liminf_{\varepsilon\rightarrow0+}I^{4}_{\varepsilon}\ge 0\,. \label{estimate 17}%
\end{equation}

From (\ref{estimate 10}) and (\ref{estimate 11}) we obtain
\[
\frac{\mathcal{H}_{\varepsilon}(w_{\varepsilon})- c_{W}r^{n-1}}{\varepsilon}
\geq-k_{0} \varepsilon^{(2-\beta)/(2\beta)}+ I^{2}_{\varepsilon}
+\delta_{\varepsilon} I^{3}_{\varepsilon}+I^{4}_{\varepsilon}\,.
\]
Since $1<\beta<2$, using (\ref{delta geq 0}), (\ref{estimate 106}),
(\ref{estimate 107}), and (\ref{estimate 17}) we conclude that
(\ref{Gamma liminf ineq 2}) holds.
\end{proof}

\section{The Limsup Inequality}

In this section we prove the following theorem.

\begin{theorem}
Let $E_{0}$ be a solution to the minimum problem (\ref{min perimeter Omega}) and let $u_0:=1-2\chi_{E_0}$.
Then there exists a sequence $\left\{  u_{\varepsilon}\right\}  \subset
H^{1}\left(  \Omega\right)  $ converging to $u_{0}$ strongly in $L^{1}(\Omega)$ and
satisfying (\ref{constraint}) such that
\begin{equation}
\limsup_{\varepsilon\rightarrow0+}
\frac{\mathcal{E}_{\varepsilon}\left(
u_{\varepsilon}\right)  - n\kappa_{\Phi} c_{W}r^{n-1}}{\varepsilon}\leq0\,. \label{gamma limsup}%
\end{equation}
\label{theorem gamma limsup}
\end{theorem}
Note that by \eqref{minimal perimeter value} the inequality \eqref{gamma limsup} is equivalent to \eqref{second gamma limsup}.

\begin{proof}
As observed in in Section \ref{s:polars}, the set $E_{0}$ has the form
$B^{\Phi^{\circ}}_{r}\!\!\left(  x_{0} \right)  $, with $B^{\Phi^{\circ}}%
_{r}\!\!\left(  x_{0} \right)  \subset\Omega$ and $r$ defined by (\ref{r}).
We recall that by hypothesis \eqref{enlarged ball} there exist $y_{0}\in\Omega$ and $\delta>0$ satisfying $B^{\Phi^{\circ}}_{r}\!\!\left(  x_0 \right)  \subset B^{\Phi^{\circ}%
}_{r+\delta}\!\left(  y_0 \right)  \subset\Omega$. 

We claim that $\Phi^{\circ}\left(  x_{0}-y_{0} \right)  \leq\delta$. If
$\Phi^{\circ}\left(  x_{0}-y_{0} \right)  =0$ then the inequality is trivial. If
not, let $\left\{  \lambda_{k} \right\}  $ be an increasing sequence
converging to $r/\Phi^{\circ}\left(  x_{0}-y_{0} \right)  $. Since
$x_{0}+\lambda_{k} \left(  x_{0}-y_{0} \right)  \in B^{\Phi^{\circ}}%
_{r}\!\!\left(  x_{0} \right)  $,  we have
$x_{0}+\lambda_{k} \left(  x_{0}-y_{0} \right)  \in B^{\Phi^{\circ}}%
_{r+\delta}\!\left(  y_{0} \right)  $, hence $\left(  1+\lambda_{k} \right)
\Phi^{\circ}\left(  x_{0}-y_{0} \right)  \leq r+\delta$. Passing to the limit
as $k\to+\infty$, we obtain $\Phi^{\circ}\left(  x_{0}-y_{0} \right)  +r \leq r+\delta$, which
implies $\Phi^{\circ}\left(  x_{0}-y_{0} \right)  \leq\delta$.

Since
$\Phi^{\circ}$ is convex and positively homogeneous of degree one, it is subadditive and so the
previous inequality gives
\begin{equation}
B^{\Phi^{\circ}}_{r+t\delta}\!\left(  x_{0} + t \left(  y_{0}-x_{0} \right)
\right)  \subset B^{\Phi^{\circ}}_{r+\delta}\!\left(  y_{0} \right)
\subset\Omega\label{inclusion}%
\end{equation}
for every $0\le t\le 1  $.

Let $z$ be the function defined by (\ref{Cauchy problem}). We recall that $z$ is odd, $|z(t)|\le 1$ for every $t\in\mathbb{R}$, and $z(t)=1$ if $t\ge \tau_W$, where $\tau_{W}$ is defined by (\ref{tau W}). Let
\begin{equation}
\hat{u}_{\varepsilon}\left(  x \right)  := z\left(  \frac{\Phi^{\circ
}\!\left(  x - x_{0} - \varepsilon\gamma\left(  y_{0}-x_{0} \right)  \right)
- r}{\varepsilon} \right)  \quad\text{for } x\in\Omega\,,
\label{hat u epsilon}%
\end{equation}
where $\gamma:= \tau_{W} /\delta$. Then $\hat{u}_{\varepsilon}\in H^{1}\left(
\Omega\right)  $ and $\hat{u}_{\varepsilon}\rightarrow u_{0}=1-2\chi
_{B^{\Phi^{\circ}}_{r}\!\left(  x_{0} \right)  } $ strongly in $L^{1}(\Omega)$. Since
$\hat{u}_{\varepsilon}=1$ on $\Omega\setminus B^{\Phi^{\circ}}_{r+\varepsilon\tau_{W}%
}\!\!\left(  x_{0} + \varepsilon\gamma\left(  y_{0}-x_{0} \right)  \right)  $,
it follows from (\ref{inclusion}) that $\hat{u}_{\varepsilon}=1$ on
$\partial\Omega$ for $\varepsilon>0$ sufficiently small. 
 
 Observe that $\hat{u}_{\varepsilon}$ may not satisfy the mass constraint in \eqref{constraint}, and so we estimate the possible error
 $$\omega_{\varepsilon}    :=\int_{\Omega}\hat{u}_{\varepsilon}\left(  x
\right)  dx- m\,. $$
Using the fact that   $\hat{u}_{\varepsilon}=-1$ on
$B^{\Phi^{\circ}}_{r-\varepsilon\tau_{W}}\!\!\left(  x_{0} + \varepsilon
\gamma\left(  y_{0}-x_{0} \right)  \right)  $, by (\ref{kappa Phi})
and (\ref{integration polar}) we get
\begin{align*}
\omega_{\varepsilon}  &  = \left\vert \Omega\setminus B^{\Phi^{\circ}}_{r+\varepsilon\tau_{W}%
}\!\!\left(  x_{0} + \varepsilon\gamma\left(  y_{0}-x_{0} \right)  \right)\right\vert 
- \left\vert B^{\Phi^{\circ}}_{r-\varepsilon\tau_{W}}\!\!\left(  x_{0} +
\varepsilon\gamma\left(  y_{0}-x_{0} \right)  \right)  \right\vert
\\
& \quad  +
n\kappa_{\Phi} \int_{r-\varepsilon\tau_{W}}^{r+\varepsilon\tau_{W}}
z\left(  \tfrac{\rho-r}{\varepsilon} \right)  \rho^{n-1}d\rho-m \\
&  =\ \left\vert \Omega\right\vert - m - \kappa_{\Phi}\! \left(  \left(
r+\varepsilon\tau_{W}\right)  ^{n} + \left(  r-\varepsilon\tau_{W}\right)
^{n} - n\varepsilon\int_{-\tau_{W}}^{\tau_{W}} z\left(  t \right)  \left(
r+\varepsilon t\right)  ^{n-1} dt \right)  .
\end{align*}
Since $z$ is odd, by  (\ref{constraint reached}) there
exists a constant $M>0$ such that
\begin{equation}
\left\vert \omega_{\varepsilon} \right\vert \leq M\varepsilon^{2}
\label{estimate omega epsilon}%
\end{equation}
for $\varepsilon>0$ sufficiently small.

We now correct $\hat{u}_{\varepsilon}$ in order to satisfy the mass constraint in \eqref{constraint}. We fix $\varphi\in C^{\infty}_{c}\left(\mathbb{R}^n
 \right)  $ with support contained in $B^{\Phi^{\circ}}_{r/2}\!\left(  x_{0} \right)$ and 
\begin{equation}
\int_{\mathbb{R}^n} \varphi\left(  x\right)   dx=1\,,
\label{integral = 1}%
\end{equation}
and we define
\begin{equation}
u_{\varepsilon}\left(  x \right)    :=\hat{u}_{\varepsilon}\left(  x
\right)  -\omega_{\varepsilon} \varphi\left(  x \right)  \,.
\label{u epsilon}%
\
\end{equation}

Taking into account the definition of $\omega_{\varepsilon}$, we find that  by
 (\ref{integral = 1}), 
$u_{\varepsilon}$ satisfies the mass constraint in (\ref{constraint}). Since $B^{\Phi^{\circ}}_{r/2}\!\left(  x_{0} \right)\subset 
 B^{\Phi^{\circ}}_{r-\varepsilon\tau_{W}}\!\!\left(  x_{0} + \varepsilon
\gamma\left(  y_{0}-x_{0} \right)  \right)$, the support of $\varphi$ is contained in $\{\hat{u}_\varepsilon=-1\}$. Hence $u_\varepsilon$  still
satisfies the boundary condition in  (\ref{constraint}) for $\varepsilon>0$ sufficiently small, 
and 
\begin{align*}
 \int_{\Omega}W(u_\varepsilon(x))\,dx
 &=\int_{\{\hat{u}_\varepsilon\ne -1\}}W(\hat{u}_\varepsilon(x))\,dx
 +\int_{\{\hat{u}_\varepsilon= -1\}}W(-1-\omega_\varepsilon\varphi(x))\,dx\\
 &=\int_{\Omega}W(\hat{u}_\varepsilon(x))\,dx
 +\int_{\Omega}W(-1-\omega_\varepsilon\varphi(x))\,dx\,,
\end{align*}
where the last equality follows from the fact that $W(-1)=0$.  
Since the supports of $\nabla \hat{u}_\varepsilon$ and $\nabla \varphi$ are disjoint, the previous equality implies that 
\begin{equation}
\mathcal{E}_{\varepsilon}\left(  u_{\varepsilon} \right)  =
 \mathcal{E}_{\varepsilon}\left(  \hat{u}_\varepsilon \right)  +\mathcal{E}_{\varepsilon}\left( -1-\omega_\varepsilon \varphi\right)\,.
\label{sum} 
\end{equation}

By (\ref{derivative polar}), (\ref{hat u epsilon}),  and
(\ref{u epsilon}) we have
\begin{equation*}
\Phi\left(  \nabla \hat{u}_{\varepsilon}\left(  x \right)  \right)  =   \frac
{1}{\varepsilon} z^{\prime}\left(  \frac{ \Phi^{\circ}\!\left(  x - x_{0} -
\varepsilon\gamma\left(  y_{0}-x_{0} \right)  \right)  - r}{\varepsilon}
\right) 
\end{equation*}
for a.e.\ $x\in\Omega$. Since  $z(t)=\pm1$ for $\pm t\geq\tau_{W}$, by
(\ref{integration polar}) and by the equality $W\left(  \pm1 \right)  =0$ we
obtain%
\begin{align}
\mathcal{E}_{\varepsilon}\left(  \hat{u}_\varepsilon \right)  &= \frac{n\kappa_{\Phi}}{\varepsilon} \int_{r-\varepsilon\tau_{W}%
}^{r+\varepsilon\tau_{W}}\left(  W\left(  z\left(  \tfrac{\rho-r}{\varepsilon}
\right)  \right)  + \left\vert z^{\prime}\left(  \tfrac{\rho-r}{\varepsilon}
\right)  \right\vert ^{2} \right)  \rho^{n-1}d\rho\nonumber\\
&  =n\kappa_{\Phi} \int_{-\tau_{W}}^{\tau_{W}}\left(  W\left(  z\left(  t
\right)  \right)  + \left\vert z^{\prime}\left(  t \right)  \right\vert ^{2}
\right)  \left(  r+\varepsilon t \right)  ^{n-1} dt \label{E first}\\
  & = \vphantom{ \int_{0}^{R/2}} n\kappa_{\Phi} c_{W} r^{n-1}+O(\varepsilon
^{2}) \,,\nonumber
\end{align}
where we used the change of variables $t= \frac{\rho-r}{\varepsilon}$ and
(\ref{constant cW}), and  in the last equality we used  (\ref{constant cW}%
), taking into account once again the fact that $t\mapsto( W\left(  z\left(  t
\right)  \right)  + \left\vert z^{\prime}\left(  t \right)  \right\vert ^{2} )
t$ is odd.

On the other hand, by \eqref{W beta} and \eqref{growth Gamma},
\begin{align}\label{E second}
 \mathcal{E}_{\varepsilon}\left( -1-\omega_\varepsilon \varphi\right)
 \le \frac{|\omega_\varepsilon|^\beta}{\varepsilon}\int_{\Omega}|\varphi|^\beta\,dx
 +C_\Phi\varepsilon|\omega_\varepsilon|^2\int_{\Omega}|\nabla\varphi|^2\,dx\,.
\end{align}
From \eqref{sum}, \eqref{E first} and 
\eqref{E second} we get
\begin{align*}
&  \frac{\mathcal{E}_{\varepsilon}\left(  u_{\varepsilon}\right)
-n\kappa_{\Phi} c_{W} r^{n-1} }{\varepsilon}\\
&  \qquad\le\frac{|\omega_\varepsilon|^\beta}{\varepsilon^2}\int_{\Omega}|\varphi|^\beta\,dx
 +C_\Phi|\omega_\varepsilon|^2\int_{\Omega}|\nabla\varphi|^2\,dx
+O(\varepsilon)\,.
\end{align*}
Recalling that $\beta>1$, from (\ref{estimate omega epsilon}) we obtain
(\ref{gamma limsup}).
\end{proof}

\section*{Acknowledgments}

The authors warmly thank the Center for Nonlinear Analysis (NSF Grant No.
 DMS-0635983), where part of this research was carried out. The
research of I. Fonseca was partially funded by the National Science Foundation
under Grant No. DMS-0905778 and that of G. Leoni under Grant No. DMS-1007989.
I. Fonseca and G. Leoni also acknowledge support of the National Science
Foundation under the PIRE Grant No. OISE-0967140. The research of G. Dal Maso
was also supported by by the Italian Ministry of Education, University, and
Research under the Project ``Variational Problems with Multiple Scales'' 2008
and by the European Research Council under Grant No. 290888 ``Quasistatic and
Dynamic Evolution Problems in Plasticity and Fracture''. 
The authors wish to thank Matteo Focardi for several discussions on the subject of this paper 
and  Michael Goldman, who called their attention to reference \cite{AFTL}.
The results of this paper led to a dramatic simplification in the proof of the
$\Gamma$-limsup inequality and were crucial in the proof of the $\Gamma
$-liminf inequality in the anisotropic case.

\end{document}